\documentclass[preprint,3p]{elsarticle}
\usepackage{stmaryrd}
\SetSymbolFont{stmry}{bold}{U}{stmry}{m}{n}
\usepackage{amsmath,bbm,mathbbol,mathtools}
\usepackage{amssymb,amsfonts,amsthm}
\usepackage{mathrsfs}
\usepackage{tabularx,lmodern}
\usepackage{color,soul,enumitem}
\usepackage{epsfig,epstopdf,color}
\usepackage{hyperref,graphicx}
\usepackage{multirow}
\usepackage[ruled]{algorithm2e}
\usepackage{flushend}
\usepackage{verbatim}
\usepackage{caption}
\usepackage{graphicx, subfig}
\usepackage{arydshln}
\usepackage[level]{datetime}
\usepackage{float}
\usepackage{booktabs}
\usepackage{tabularx}
\usepackage{natbib}
\usepackage{relsize}
\usepackage{ulem}
\usepackage{setspace}
\usepackage{epstopdf}
\usepackage{appendix}
\usepackage{bm}
\usepackage{algpseudocode}
\usepackage{adjustbox}
\usepackage{array}
\usepackage{empheq}

\numberwithin{equation}{section}

\usepackage{hyperref}
\hypersetup{
	colorlinks=true,
	citecolor=blue,
	linkcolor=red,
	urlcolor=green}
	
\usepackage{calc}     
\newcommand{\custombar}[2][1]{
	\stackrel{\rule{#1em}{0.4pt}}{#2}
}

\makeatletter

\newcommand{\Rmnum}[1]{\expandafter\@slowromancap\romannumeral #1@}
\makeatother

\newcommand{\be}{\begin{equation}}
\newcommand{\ee}{\end{equation}}
\newcommand{\ba}{\begin{array}}
\newcommand{\ea}{\end{array}}
\newcommand{\bea}{\begin{eqnarray}}
\newcommand{\eea}{\end{eqnarray}}
\newcommand{\beas}{\begin{eqnarray*}}
\newcommand{\eeas}{\end{eqnarray*}}

\newcommand{\im}{i}

\newtheorem{exmp}{Example}
\newtheorem{remark}{Remark}[section]

\newtheorem{lem}{Lemma}
\newtheorem{prf}{proof}

\begin{document}
	
\title{An efficient compact splitting Fourier spectral methods for computing the dynamics of rotating spin-orbit coupled spin-2 Bose-Einstein condenstates}

\author[eit]{Xin Liu}
\ead{xinliu@eitech.edu.cn}

\author[hunnu]{Ziqing Xie}
\ead{ziqingxie@hunnu.edu.cn}

\author[hunnu]{Yongjun Yuan}
\ead{yyj1983@hunnu.edu.cn}

\author[tju1,tju2]{Yong Zhang}
\ead{Zhang\_Yong@tju.edu.cn}

\author[hunnu]{Xinyi Zhao\corref{5}}
\ead{xinyizhao@hunnu.edu.cn}

\address[eit]{Eastern Institute for Advanced Study, Eastern Institute of Technology, Ningbo,  315200, China}
\address[hunnu]{MOE-LCSM, School of Mathematics and Statistics, Hunan Normal University, Changsha, Hunan 410081, China}

\address[tju1]{Center for Applied Mathematics and KL-AAGDM, Tianjin University, Tianjin, 300072, China}
\address[tju2]{State Key Laboratory of Synthetic Biology,   Tianjin University, Tianjin, 300072, China}

\cortext[5]{Corresponding author.}
	
\begin{abstract}

    This paper investigates the dynamics of spin-2 Bose-Einstein condensates (BECs) with rotation and spin-orbit coupling (SOC).
    In order to better simulate the dynamics, we present an efficient high-order compact splitting Fourier spectral method. This method splits the Hamiltonian into a linear part,
    which consists of the Laplace, rotation and SOC terms, and a nonlinear part that includes all the remaining terms. The wave function is well approximated by the Fourier spectral method and is numerically accessed with discrete Fast Fourier transform (FFT).
    For linear subproblem, the handling of rotation term and SOC term poses a major challenge. Using a function mapping based on rotation, we can integrate the linear subproblem exactly and explicitly. This mapping we propose not only helps eliminate the rotation term, but also prevents the SOC term from evolving into a time-dependent form. The nonlinear subproblem is integrated
    analytically in physical space.
    Such ``compact'' splitting involves only two operators and facilitates the design of high-order splitting schemes. Our method is spectrally accurate in space and high order in time. It is efficient, explicit, unconditionally stable and simple to implement. In addition, we derive some dynamical properties and carry out a systematic study, including accuracy and efficiency tests, dynamical property verification, the SOC effects and dynamics of vortex lattice.

\end{abstract}


\maketitle
\tableofcontents
	
	\section{Introduction}

    Bose-Einstein condensate (BEC) is a gaseous, superfluid state of matter that forms when the dilute boson gas is cooled to temperatures near absolute zero. In 1995, a team led by Eric Cornell and Carl Wieman combined laser and evaporative cooling techniques to achieve the first observation of BEC in the dilute $^{87}$Rb gases \cite{A1995}.
    Early experiments used magnetic traps, which froze the spin degrees of the atoms. In 1998, a BEC with internal spin degrees of freedom, known as a spinor BEC, was first realized in an optical dipole trap \cite{S1998}. In optical traps, due to the interparticle interaction, the direction of atomic spins can change. Consequently, the order parameter of a spin‑$F$ BEC has $2F + 1$ components, which can vary in space and time, giving rise to rich spin textures.
    Recently, the spin-orbit coupling (SOC) , which plays a key role in spintronic devices \cite{K2009}, spin Hall effect \cite{K2004}, topological insulators \cite{H2010} and Majorana fermions \cite{W2009}, has been successfully induced in a neutral atomic BECs by dressing two atomic spin states with a pair of lasers. These experiments sparked a strong activity in the area of spin-orbit-coupled cold atoms and discovered a number of exciting phenomena.
    In particular, many investigations show that the combination of SOC, rotation and atomic intrinsic interactions can generate various phenomena in the spinor BEC \cite{J2008,L2012,Z2013}. For example, a variety of exotic spin textures and fractional quantized vortices are exhibited for rotating spin-orbit-coupled BECs \cite{Z2011}, a new necklace-type state with double-ring structure is created in a spin-1 BEC system due to the SOC and rotation, and the ground-state phase diagrams of the spin-orbit-angular-momentum coupled $^{87}$Rb condensate are experimentally mapped out with first-order phase transitions \cite{Z2019}, SOC can control topological vortical phase transitions characterized by Anderson-Toulouse vortices, vortex-dipole lattices, and perpendicular vortex chains in the ground states of spin-2 BEC with SOC and rotation \cite{Z2021}.


	In the mean-field regime, when the temperature $T$ is much lower than the critical temperature $T_c$, a rotating SOC spin-2 BEC can be described by the five-component wave functions $\Psi:=\Psi (\mathbf{\textbf{x}} ,t)=\left( \psi_{2}(\mathbf{\textbf{x}},t),\psi_{1}(\mathbf{\textbf{x}},t),\psi_{0}(\mathbf{\textbf{x}},t),\psi_{-1}(\mathbf{\textbf{x}},t),\psi_{-2}(\mathbf{\textbf{x}},t) \right)^{\top}$ satisfying the coupled Gross-Pitaevskii equations (CGPEs) \cite{LY2025,W2019,X2019}
	\begin{align}
		\im\partial_t\psi_{\pm2}=&\mathcal{H}_0\psi_{\pm2}+\pm2c_1F_z\psi_{\pm2}+c_1F_{\mp}\psi_{\pm1}+\frac{c_2}{\sqrt{5}}A_{00}\bar{\psi}_{\mp2}
		{-\gamma L_{\mp}\psi_{\pm1}},\label{CGPEs1}\\
		\im\partial_t\psi_{\pm1}=&\mathcal{H}_0\psi_{\pm1}+\pm c_1F_z\psi_{\pm1}+c_1\left(\frac{\sqrt{6}}{2}F_{\mp}\psi_0+F_{\pm}\psi_{\pm2}\right)-\frac{c_2}{\sqrt{5}}A_{00}\bar{\psi}_{\mp1}
		{-\gamma[L_{\pm}\psi_{\pm2}+\frac{\sqrt{6}}{2}L_{\mp}\psi_0]},\\
		\im\partial_t\psi_0=&\mathcal{H}_0\psi_0+\frac{\sqrt{6}}{2}c_1\left(F_+\psi_1+F_-\psi_{-1}\right)+\frac{c_2}{\sqrt{5}}A_{00}\bar{\psi}_0
		{-\frac{\sqrt{6}}{2}\gamma\left[L_{+}\psi_1+L_{-}\psi_{-1}\right]},\label{CGPEs3}
	\end{align}
    where $\mathcal{H}_0=\left(-\frac{1}{2}\Delta+V+c_0\rho-\Omega L_z\right)$. Here $t$ denotes time and $\mathbf{x}=(x,y)^{\top}\in \mathbb{R}^2$ or $\mathbf{x}=(x,y,z)^{\top}\in \mathbb{R}^3$ is the Cartesian coordinate vector. $\bar{f}$ represents the conjugate of complex number $f$. $c_0$, $c_1$ and $c_2$ are the dimensionless mean-field, spin-exchange interaction constant and spin singlet-pairing interaction contants respectively. $\Delta$ is the Laplacian operator, and $\rho=\sum_{\ell=-2}^{2}\left|\psi_{\ell}\right|^2$ is the total density. The constant $\Omega$ is the angular velocity, and $L_z=-\im\left(x\partial_y-y\partial_x\right)$ is the z-component of the angular momentum operator. The constant $\gamma$ is the coupling strength, and $L_{\pm}:=\im\partial_y\pm\partial_x$ are the spin-orbit coupling operators. $V=V(\mathbf{x})$ is a real-valued external trapping potential. Here we choose $V(\mathbf{x})$ as a harmonic potential, i.e.,
    \begin{equation}\label{harmonic}
    	V(\mathbf{x}) = \frac{1}{2}
    	\begin{cases}
    		\gamma_x^2 x^2 + \gamma_y^2 y^2, & d = 2, \\
    		\gamma_x^2 x^2 + \gamma_y^2 y^2 + \gamma_z^2 z^2, & d = 3,
    	\end{cases}
    \end{equation}
	where $\gamma_x$, $\gamma_y$, and $\gamma_z$ represent the trap frequencies in the respective $x$-, $y$-, and $z$-directions. The spin vector $\mathbf{F}:=\left(F_x(\Psi),F_y(\Psi),F_z(\Psi)\right)^{\top}:=\left(\Psi^{\mathsf{H}}f_x\Psi,\Psi^{\mathsf{H}}f_y\Psi,\Psi^{\mathsf{H}}f_z\Psi\right)^{\top}$ with the conjugate transpose $\Psi^{\mathsf{H}}$ of $\Psi$ and the spin-2 matrices $\mathbf{f} = \left(f_x, f_y, f_z\right)^{\top}$, where
	\begin{equation*}
		f_x=\begin{pmatrix}
			0&1&0&0&0\\
			1&0&\frac{\sqrt{6}}{2}&0&0\\
			0&\frac{\sqrt{6}}{2}&0&\frac{\sqrt{6}}{2}&0\\
			0&0&\frac{\sqrt{6}}{2}&0&1\\
			0&0&0&1&0
		\end{pmatrix},f_y=\im\begin{pmatrix}
			0&-1&0&0&0\\
			1&0&-\frac{\sqrt{6}}{2}&0&0\\
			0&\frac{\sqrt{6}}{2}&0&-\frac{\sqrt{6}}{2}&0\\
			0&0&\frac{\sqrt{6}}{2}&0&-1\\
			0&0&0&1&0
		\end{pmatrix},f_z=\begin{pmatrix}
			2&0&0&0&0\\
			0&1&0&0&0\\
			0&0&0&0&0\\
			0&0&0&-1&0\\
			0&0&0&0&-2
		\end{pmatrix}.
	\end{equation*}
	Furthermore, $A_{00}:=\Psi^{\top}A\Psi$ with the matrix $A$ taking the form $A_{ij}=\frac{1}{\sqrt{5}}(-1)^{i-1}\delta_{i+j,6}~(i,j=1,2,\cdots,5)$, where $\delta_{i+j,6}$ represents the Kronecker delta.
	To be more detailed, we have
	\begin{align*}
		&F_x=\bar{\psi}_{1}\psi_{2}+\bar{\psi}_{2}\psi_{1}+\bar{\psi}_{-2}\psi_{-1}+\bar{\psi}_{-1}\psi_{-2}
		+\frac{\sqrt{6}}{2}\left(\bar{\psi}_{0}\psi_{1}+\bar{\psi}_{1}\psi_{0}+\bar{\psi}_{-1}\psi_{0}+\bar{\psi}_{0}\psi_{-1}\right),\\    			
		&F_y=\im\left[\bar{\psi}_{1}\psi_{2}-\bar{\psi}_{2}\psi_{1}+\bar{\psi}_{-2}\psi_{-1}-\bar{\psi}_{-1}\psi_{-2}
		+\frac{\sqrt{6}}{2}\left(\bar{\psi}_{0}\psi_{1}-\bar{\psi}_{1}\psi_{0}+\bar{\psi}_{-1}\psi_{0}-\bar{\psi}_{0}\psi_{-1}\right)\right],\\
		&F_z=2\left|\psi_2\right|^2+\left|\psi_1\right|^2-\left|\psi_{-1}\right|^2-2\left|\psi_{-2}\right|^2,~A_{00}=\dfrac{1}{\sqrt{5}}\left(2\psi_2\psi_{-2}-2\psi_1\psi_{-1}+\psi_0^2\right).
	\end{align*}
	Notice that all spin-2 matrices are Hermitian and the spin vector entries are real numbers.
	The CGPEs \eqref{CGPEs} can be written in the compact form as
	\begin{equation}\label{CGPEs}
		\im\partial_{t}\Psi=\left(\mathcal{H}_0I_5+c_1\mathbf{F}\cdot\mathbf{f}-\gamma\mathcal{S}\right)\Psi+ c_2 A_{00}A\bar{\Psi},
	\end{equation}
    where $I_5$ is the $5\times5$ identity matrix and
	\begin{center}
			$\mathcal{S}=$
			$\begin{pmatrix}
				0 & L_- & 0 & 0 & 0\\
				L_+ & 0 & \frac{\sqrt{6}}{2}L_- & 0 & 0\\
				0 & \frac{\sqrt{6}}{2}L_+ & 0 & \frac{\sqrt{6}}{2}L_- & 0\\
				0 & 0 & \frac{\sqrt{6}}{2}L_+ & 0 & L_-\\
				0 & 0 & 0 & L_+ & 0
			\end{pmatrix}$,
			$\mathbf{F}\cdot\mathbf{f}=$
			$\begin{pmatrix}
				2F_z & F_- & 0 & 0 & 0\\
				F_+ & F_z & \frac{\sqrt{6}}{2}F_- & 0 & 0\\
				0 & \frac{\sqrt{6}}{2}F_+ & 0 & \frac{\sqrt{6}}{2}F_- & 0\\
				0 & 0 & \frac{\sqrt{6}}{2}F_+ & -F_z & F_-\\
				0 & 0 & 0 & F_+ & -2F_z
			\end{pmatrix}$,
	\end{center}
	with $F_+=F_x+\im F_y=2\left(\bar{\psi}_{2}\psi_{1}+\bar{\psi}_{-1}\psi_{-2}\right)+\sqrt{6}\left(\bar{\psi}_1\psi_0+\bar{\psi}_0\psi_{-1}\right)$ and $F_-=F_x-\im F_y=\bar{F}_+.$
	
	There have been many numerical methods proposed for studying the dynamics of single-component BEC \cite{A2013,BC2013-2,BJ2003,BZ2005,C2000,M2009,T2006}, among which the time-splitting sine/Fourier pseudospectral method is one of the most successful. The time-splitting pseudospectral method exhibits spatial spectral-order accuracy and is relatively straightforward to implement. As far as we known, this method has been generalized to study the dynamics of spinor BECs \cite{BC2018}, nonrotating/rotating spin‑1 BEC with SOC term \cite{BK2022,C2025,K2021,LY2025}, and nonrotating spin‑2 BEC with SOC term \cite{BK2022}. There are little research on dynamic simulations for spin-2 BEC with SOC term and the more physically interesting rotating case. In this paper, we aim to perform a comprehensive study of the dynamics for rotating SOC spin-2 BEC.
	
	Numerically, the most challenges lie in proper treatments of the rotation and SOC terms. There exist some methods that successfully handle the rotating term. The standard alternating direction implicit (standard ADI) method splits the Hamiltonian into three parts and is of second order accurate in time \cite{BW2006}. However, it is somewhat tedious and complicated to construct a high-order scheme with such triple operators. The exact splitting method (ESM) groups the Laplace and rotation terms into a linear part, enabling its corresponding subproblem to be integrated analytically and explicitly in Fourier space \cite{LM2025}. this property facilitates the construction of high‑order temporal schemes. However, extending or adapting ESM to spinor BEC with SOC term remains nontrivial and highly challenging. The rotating Lagrangian coordinates (RLC) method eliminates the rotation term, thus simplifying the construction of high-order schemes \cite{BM2013}. However, when real-time dynamics are
	required, it is imperative to rotate the wave function $\psi(\mathbf{x},t)$ from rotating Lagrangian coordinates to physical Cartesian coordinatesat each time step $t_n$, and such rotation mapping is quite exhaustive computationally and poses great challenges to the simulation efficiency. Under these circumstances, Liu et al. \cite{LY2025} proposed a method that not only borrows the RLC idea to handle the rotation term, but also effectively addresses the complexities introduced by the SOC term.
	
	In \cite{LY2025}, by rotating the wave function in a fixed rectangular domain, the rotational term is eliminated for the rotated wave function. However, the SOC term thereby becomes explicitly time-dependent. To handle this, the linear system is integrated in Fourier space via a complicated time-dependent matrix decomposition. Notably, Liu et al. proposed a Rotation‑Shear‑Decomposition‑Acceleration (RSDA) method to implement the function-rotation, replacing the conventional Fourier spectral interpolation approach. The RSDA can reformulate the rotation mapping into a computationally friendly form by useing three-shear decomposition of the rotation matrix and its equivalent PDE reformulation. Meanwhile, since RSDA can be implemented efficiently with one-dimensional FFT/iFFT in Fourier space, it results in a significant efficiency enhancement.
	
	To efficiently simulate the dynamics of rotating SOC spin-2 BEC, we introduce a high-order compact splitting Fourier spectral method. To be specific, the CGPEs \eqref{CGPEs} is split into a linear subproblem
	\begin{equation}
		\im\partial_{t}\Psi(\mathbf{x},t)=\left[(-\frac{1}{2}\Delta-\Omega L_z)I_5-\gamma\mathcal{S}\right]\Psi(\mathbf{x},t) + c_2 A_{00}A\bar{\Psi}(\mathbf{x},t):=\mathcal{A}\Psi(\mathbf{x},t),
	\end{equation}
	and a nonlinear subproblem
	\begin{equation}
		\im\partial_{t}\Psi(\mathbf{x},t)=\left[(V+c_0\rho)I_5+c_1\mathbf{F}\cdot\mathbf{f}\right]\Psi(\mathbf{x},t) + c_2 A_{00}A\bar{\Psi}(\mathbf{x},t):=\mathcal{B}\Psi(\mathbf{x},t).
	\end{equation}
	The nonlinear subproblem is integrated analytically in physical space. Inspired by the idea of \cite{LY2025}, we propose a new function mapping that better addresses linear subproblem.
	
	The function mapping proposed in the \cite{LY2025} turns SOC term into time-dependent, that brings difficulties to solving linear subproblem. Therefore, based on our observations, we propose in this paper a new function mapping \eqref{mapping}. This mapping not only eliminates the rotational term but also prevents the SOC term from becoming time-dependent. Although it introduces an additional term, this term is completely explicit and linear, and thus poses no difficulty to the solution process.
	
	The organization of this paper is as follows. In Section \ref{sec-dp}, we derive some dynamical laws for some physical quantities. In Section \ref{sec-nm}, we propose an efficient and robust splitting Fourier spectral method to simulate the dynamics, and prove the mass conservation (unconditionally stability) and magnetization conservation. In Section \ref{sec-nr}, we test the temporal/spatial accuracies and efficiency, and study some interesting numerical phenomena. We make some concluding remarks in Section \ref{sec-c}.

	\section{Dynamical properties}\label{sec-dp}
    In this section, we demonstrate some main quantities in the study of dynamics of spin-2 BEC with rotation and SOC, including mass, energy, magnetization, angular momentum expectation and condensate width. The dynamical laws of these quantities are briefly presented and can be used as benchmarks for testing our numerical methods.

	\textbf{Mass and energy.} The CGPEs \eqref{CGPEs} have two important invariants: the mass of the wave function, which is defined as
	\begin{equation}
		\mathcal{N}(t) := \mathcal{N}(\Psi(\cdot,t)) := \int_{\mathbb{R}^d} \sum_{\ell=-2}^{2} \left| \psi_{\ell}(\mathbf{x},t) \right|^2 \,\mathrm{d}\mathbf{x} \equiv \mathcal{N}(\Psi(\cdot,0)) \equiv 1, \quad t \geq 0,
	\end{equation}
	and the energy per particle
	\begin{align}
		\mathcal{E}(t) := \mathcal{E}(\Psi(\cdot,t)) &= \int_{\mathbb{R}^d} \sum_{\ell=-2}^{2} \Big( \frac{1}{2} |\nabla \psi_{\ell}|^2 + V(\mathbf{x}) |\psi_{\ell}|^2 - \Omega \bar{\psi}_{\ell} L_z \psi_{\ell} \Big) + \frac{c_0}{2} \rho^2 + \frac{c_1}{2} \left(|F_+|^2 + |F_z|^2\right) \nonumber\\
		& + \frac{c_2}{2} |A_{00}|^2 - \gamma\Bigg[ \bar{\psi}_2L_- \psi_1 + \bar{\psi}_1(L_+ \psi_2 + \frac{\sqrt{6}}{2} L_- \psi_0) + \frac{\sqrt{6}}{2}\bar{\psi}_0(L_+ \psi_1 + L_- \psi_{-1}) \nonumber\\
		&+ \bar{\psi}_{-1}(L_- \psi_{-2} + \frac{\sqrt{6}}{2} L_+ \psi_0) + \bar{\psi}_{-2} L_+ \psi_{-1} \Bigg] \,\mathrm{d}\mathbf{x} \nonumber\\
		& \equiv \mathcal{E}(\Psi(\cdot,0)).
	\end{align}
	

    \textbf{Magnetization.} The magnetization of the wave function, defined as
    \begin{equation}
    	\mathcal{M}(t) = \mathcal{M}(\Psi(\cdot,t)) := \sum_{\ell=-2}^{2} \int_{\mathbb{R}^d} \ell\left| \psi_{\ell}(\mathbf{x},t) \right|^2 \,\mathrm{d}\mathbf{x},
    \end{equation}
    satisfies the relation
    \begin{equation}
    	\frac{d}{dt}\mathcal{M}(t) = 2\gamma\,\Re \int_{\mathbb{R}^d} \im \Big( \bar{\psi}_2L_-\psi_1 - \psi_{-2}L_-\bar{\psi}_{-1} + \frac{\sqrt{6}}{2}\bar{\psi}_1L_-\psi_0 - \frac{\sqrt{6}}{2} \psi_{-1}L_-\bar{\psi}_0 \Big) \mathrm{d}\mathbf{x}.
    \end{equation}
    Thus the magnetization is conserved (i.e., $\mathcal{M}(t) \equiv \mathcal{M}(0)$, $t\geq0$) in the absence of spin-orbit coupling ($\gamma = 0$).
	
	\textbf{Angular momentum expectation.} The angular momentum expectation of the condensate \eqref{CGPEs} is defined as
	\begin{equation}
		\langle L_z \rangle(t) = \sum_{\ell=-2}^{2} \int_{\mathbb{R}^d} \bar{\psi}_{\ell}(\mathbf{x},t) L_z \psi_{\ell}(\mathbf{x},t) \,\mathrm{d}\mathbf{x}, \quad t \geq 0.
	\end{equation}
	
%
	\begin{lem}\label{Ame}
		For CGPEs \eqref{CGPEs} with harmonic potential \eqref{harmonic}, the dynamics of angular momentum expectation is governed by the following
		ordinary differential equation (ODE)
		\begin{equation*}
			\frac{d\langle L_z \rangle(t)}{dt} = (\gamma_x^2 - \gamma_y^2) \int_{\mathbb{R}^d} xy~\rho~ \mathrm{d}\mathbf{x} - 2\gamma\Re \int_{\mathbb{R}^d} \im \Big( \bar{\psi}_{-1}L_-\psi_{-2} - \bar{\psi}_1L_+\psi_2 + \frac{\sqrt{6}}{2}\bar{\psi}_1L_-\psi_0 - \frac{\sqrt{6}}{2}\bar{\psi}_{-1}L_+\psi_0 \Big) \mathrm{d}\mathbf{x}.
		\end{equation*}
		Thus the angular momentum expectation is conserved, i.e.,
		\begin{equation}
			\langle L_z \rangle(t) \equiv \langle L_z \rangle(0), \quad t \geq 0,
		\end{equation}
		when $\gamma_x = \gamma_y$ and $\gamma = 0$.
	\end{lem}
	\begin{prf}
		By using \eqref{CGPEs} and performing integration by parts, we can derive
		\begin{align*}
			\frac{d \langle L_z \rangle(t)}{dt} = & \sum_{\ell=-2}^{2} \int \left[ (\partial_t\bar{\psi}_{\ell})\,(L_z \psi_{\ell}) + \bar{\psi}_{\ell}\,L_z(\partial_t\psi_{\ell}) \right] \,\mathrm{d}\mathbf{x} \\
			= & \int_{\mathbb{R}^d} V(\mathbf{x}) \left(x\partial_y-y\partial_x\right)\rho - \gamma \im \Big( \psi_{-1}L_+\bar{\psi}_2 + \bar{\psi}_{-1}L_-\psi_2 - \psi_1L_-\bar{\psi_2} - \psi_1L_+\bar{\psi_2}\\
			& + \frac{\sqrt{6}}{2}\left[ \psi_1L_+\bar{\psi}_0 + \bar{\psi}_1L_-\psi_0 - \psi_{-1}L_-\psi_0 - \psi_{-1}L_+\psi_0\right] \Big) \,\mathrm{d}\mathbf{x} \\
			= & (\gamma_x^2 - \gamma_y^2) \int_{\mathbb{R}^d} xy\,\rho\, \mathrm{d}\mathbf{x} - 2\gamma\Re \int_{\mathbb{R}^d} \im \Big( \bar{\psi}_{-1}L_-\psi_{-2} - \bar{\psi}_1L_+\psi_2 + \frac{\sqrt{6}}{2}\bar{\psi}_1L_-\psi_0 - \frac{\sqrt{6}}{2}\bar{\psi}_{-1}L_+\psi_0 \Big) \mathrm{d}\mathbf{x}.
		\end{align*}
		This completes the proof.
	\end{prf}
%
	
	\textbf{Condensate width.} The condensate width in the $\alpha$-direction (where $\alpha=x,y,z$) is defined as $\sigma_{\alpha}=\sqrt{\delta_{\alpha}(t)},t \geq 0$, where
	\begin{equation}\label{ConWid}
		\delta_{\alpha}(t)=\sum_{\ell=-2}^{2}\delta_{\alpha,\ell}(t) \quad \text{with} \quad \delta_{\alpha,\ell}=\int_{\mathbb{R}^d}\alpha^{2}|\psi_{\ell}(\mathbf{x},t)|^{2}\,\mathrm{d}\mathbf{x}.
	\end{equation}
	In particular, the following lemma describes its dynamics in the 2D case.

	\begin{lem}\label{Cw}
		For the 2D CGPEs \eqref{CGPEs} with the radially symmetric harmonic potential, i.e., \eqref{harmonic} with $\gamma_x = \gamma_y =: \gamma_r$, we have
		\begin{equation}\label{d^2delta}
			\frac{d^{2}\delta_{r}(t)}{dt^{2}} = -4\gamma_r^2\delta_{r}(t) + 4\mathcal{E}(0) + 4\Omega\langle L_z\rangle(t) + G(\gamma,\Psi),
		\end{equation}
		where $\delta_{r}(t) := \delta_x(t) + \delta_y(t)$ and
		\begin{align}
			G(\gamma,\Psi) = 2\gamma\Re\int_{\mathbb{R}^d}&\Omega (\im y - x)F_{+}+\gamma\Big(2\rho + 3\left|\psi_1\right|^2 + 4\left|\psi_0\right|^2 + 3\left|\psi_{-1}\right|^2 + 2\bar{\psi}_2L_z\psi_2 + \bar{\psi}_1L_z\psi_1 \nonumber\\
			& - \bar{\psi}_{-1}L_z\psi_{-1} - 2\bar{\psi}_{-2}L_z\psi_{-2}\Big)\,\mathrm{d}\mathbf{x}.
		\end{align}
		In particular, when $\gamma=0$, it follows that
		\begin{equation}\label{delta_r}
			\delta_{r}(t) = \frac{\mathcal{E}(0)+\Omega\langle L_z\rangle(0)}{\gamma_r^2}\left[1 - \cos(2\gamma_{r}t)\right] + \delta_{r}^{(0)}\cos(2\gamma_{r}t) + \frac{\delta_{r}^{(1)}}{2\gamma_{r}}\sin(2\gamma_{r}t)
		\end{equation}
		with $\delta_{r}^{(0)}=\delta_{x}(0)+\delta_{y}(0)$ and $\delta_{r}^{(1)}=\dot{\delta}_{x}(0)+\dot{\delta}_{y}(0)$. Furthermore, if the initial condition is radially symmetric, we have
		\[
		\delta_{x}(t)=\delta_{y}(t)=\frac{1}{2}\delta_{r}(t), \quad t\geq0.
		\]
		Thus, the condensate widths $\sigma_{x}(t)$ and $\sigma_{y}(t)$ are periodic functions with frequency that is twice the trapping frequency $\gamma_r$ in this case.
	\end{lem}
	\begin{prf}
		By using \eqref{CGPEs}, \eqref{ConWid} and performing integration by parts, we can derive
		\begin{align*}
			\frac{d\delta_{\alpha,2}(t)}{dt} & = \int_{\mathbb{R}^d} \alpha^2 \partial_t |\psi_2|^2 \,\mathrm{d}\mathbf{x} = \int_{\mathbb{R}^d} \alpha^2 \left( \partial_t \psi_2 \bar{\psi}_2 + \psi_2 \partial_t \bar{\psi}_2 \right) \,\mathrm{d}\mathbf{x}\\
			& = 2 \Re \int_{\mathbb{R}^d} \im \alpha \Big( \psi_2 \partial_\alpha \bar{\psi}_2 - \Omega |\psi_2|^2 L_z \alpha + c_1 \alpha F_+ \bar{\psi}_1\psi_2 + \frac{c_2}{\sqrt{5}} \alpha \bar{A_{00}} \psi_{-2}\psi_2 + \gamma\alpha \psi_2L_+\bar{\psi}_1 \Big) \,\mathrm{d}\mathbf{x}.
		\end{align*}
		Analogously, we have
		\begin{align*}
			\frac{d\delta_{\alpha,\pm1}(t)}{dt} = & 2 \Re \int_{\mathbb{R}^d} \im \alpha \Big( \psi_{\pm1} \partial_\alpha \bar{\psi}_{\pm1} - \Omega |\psi_{\pm1}|^2 L_z \alpha + c_1 \alpha \big[\frac{\sqrt{6}}{2}F_{\pm}\bar{\psi}_0\psi_{\pm1} +F_{\mp}\bar{\psi}_{\pm2}\psi_{\pm1}\big] - \frac{c_2}{\sqrt{5}} \alpha \bar{A_{00}} \psi_{\mp1}\psi_{\pm1}\\
			& + \gamma\alpha \big[\psi_{\pm1}L_{\mp}\bar{\psi}_{\pm2} + \frac{\sqrt{6}}{2}\psi_{\pm1}L_{\pm}\bar{\psi}_0\big] \Big) \,\mathrm{d}\mathbf{x},\\
			\frac{d\delta_{\alpha,0}(t)}{dt} = & 2 \Re \int_{\mathbb{R}^d} \im \alpha \Big( \psi_0 \partial_\alpha \bar{\psi}_0 - \Omega |\psi_0|^2 L_z \alpha + \frac{\sqrt{6}}{2}c_1 \alpha \left[F_-\bar{\psi}_1\psi_0 + F_+\bar{\psi}_{-1}\psi_0\right] + \frac{c_2}{\sqrt{5}} \alpha \bar{A_{00}} \psi_0^2\\
			& + \frac{\sqrt{6}}{2}\gamma\alpha \left[\psi_0L_-\bar{\psi}_1 + \psi_0L_+\bar{\psi}_{-1}\right] \Big) \,\mathrm{d}\mathbf{x},\\
			\frac{d\delta_{\alpha,-2}(t)}{dt} = & 2 \Re \int_{\mathbb{R}^d} \im \alpha \Big( \psi_{-2} \partial_\alpha \bar{\psi}_{-2} - \Omega |\psi_{-2}|^2 L_z \alpha + c_1 \alpha F_- \bar{\psi}_{-1}\psi_{-2} + \frac{c_2}{\sqrt{5}} \alpha \bar{A_{00}} \psi_2\psi_{-2} + \gamma\alpha \psi_{-2}L_-\bar{\psi}_{-1} \Big) \,\mathrm{d}\mathbf{x}.
		\end{align*}
        Then, a straightforward rearrangement shows
		\begin{align*}
			\frac{d\delta_{\alpha}(t)}{dt}=&\sum_{\ell=-2}^{2}\frac{d\delta_{\alpha,\ell}(t)}{dt}=2\Re\int_{\mathbb{R}^d}\im\alpha\Bigg(\sum_{\ell=-2}^{2}(\psi_\ell\partial_\alpha\bar{\psi}_\ell - \Omega|\psi_\ell|^2L_z\alpha) + \gamma\alpha\Big[\bar{\psi}_2L_-\psi_1 + \psi_1L_-\bar{\psi}_2 + \bar{\psi}_{-1}L_-\psi_{-2}\\
			& + \psi_{-2}L_-\bar{\psi}_{-1} + \frac{\sqrt{6}}{2}\left(\bar{\psi}_1L_-\psi_0 + \psi_0L_-\bar{\psi}_1 + \bar{\psi}_0L_-\psi_{-1} + \psi_{-1}L_-\bar{\psi}_0\right)\Big]\Bigg)\,\mathrm{d}\mathbf{x}.
		\end{align*}
		By performing integration by parts, we have
		$$2\Re\int_{\mathbb{R}^d} \im\gamma\alpha^2 \psi_1L_-\bar{\psi}_2 \,\mathrm{d}\mathbf{x} = - 2\Re\int_{\mathbb{R}^d} \im\gamma\bar{\psi}_2 L_-(\alpha^2\psi_1) \,\mathrm{d}\mathbf{x} = - 2\Re\int_{\mathbb{R}^d} \im\gamma\left( \alpha^2 \bar{\psi}_2L_-\psi_1 + 2\alpha \bar{\psi}_2\psi_1L_-\alpha \right) \, \mathrm{d}\mathbf{x}.$$
		Hence,
		\begin{align*}
			\frac{d\delta_{\alpha}(t)}{dt} = & 2 \Re \int_{\mathbb{R}^d} \im\alpha \Bigg[ \sum_{\ell = -2}^{2} \left( \psi_{\ell} \partial_{\alpha} \bar{\psi}_{\ell} - \Omega |\psi_{\ell}|^2 L_z \alpha \right) - \gamma \big( 2\bar{\psi}_2\psi_1L_-\alpha + 2\bar{\psi}_{-1}\psi_{-2} L_-\alpha + \sqrt{6}\bar{\psi}_1\psi_0L_-\alpha\\
			& + \sqrt{6}\bar{\psi}_0\psi_{-1} L_-\alpha\big) \Bigg] \,\mathrm{d}\mathbf{x}.
		\end{align*}
		By further differentiating the above equation with respect to $t$ and suppressing the elementary but tedious computations, we obtain
		\begin{align}
			\frac{d^2 \delta_{\alpha}(t)}{dt^2} = & -2 \gamma_r^2 \delta_{\alpha}(t) + \int_{\mathbb{R}^d} 2 \sum_{\ell=-2}^{2} \left| \partial_{\alpha} \psi_{\ell} \right|^2 + c_0 \rho^2 + c_1 |F_z|^2 + c_1 |F_+|^2 + c_2 |A_{00}|^2 - 2\gamma \Big( \bar{\psi}_2L_-\psi_1\nonumber\\
			& + \bar{\psi}_1(L_+\psi_2+\frac{\sqrt{6}}{2}L_-\psi_0) + \frac{\sqrt{6}}{2}\bar{\psi}_0(L_+\psi_1+L_-\psi_{-1}) + \bar{\psi}_{-1}(L_-\psi_{-2}+\frac{\sqrt{6}}{2}L_+\psi_0) + \bar{\psi}_{-2}L_+\psi_{-1} \Big)\nonumber\\
			& + 2\Omega(\partial_y - \partial_x)\alpha \left( 2i \sum_{\ell=-2}^{2}\left(\bar{\psi}_{\ell} (x\partial_y + y\partial_x)\psi_{\ell}\right) + \Omega (x^2 - y^2) \rho - \gamma\left(xF_x-yF_y\right) \right) \,\mathrm{d}\mathbf{x}+ G_1(\alpha, \gamma, \Psi),\label{d^2 delta}
		\end{align}
		where
		\begin{align*}
			G_1(\alpha, \gamma, \Psi) = & 2\gamma\Re \int_{\mathbb{R}^d}2\alpha \Big( L_+\bar{\psi}_1 \partial_{\alpha}\psi_2 + (L_-\bar{\psi}_2+\frac{\sqrt{6}}{2}L_+\bar{\psi}_0)\partial_\alpha\psi_1 + \frac{\sqrt{6}}{2}(L_-\bar{\psi}_1+L_+\bar{\psi}_{-1})\partial_\alpha\psi_0\\
			& + (L_+\bar{\psi}_{-2}+\frac{\sqrt{6}}{2}L_-\bar{\psi}_0)\partial_\alpha\psi_{-1} + L_-\bar{\psi}_{-1} \partial_{\alpha}\psi_{-2} \Big) + L_-\alpha \Bigg( 2\psi_1\partial_\alpha\bar{\psi}_2 + 2\psi_{-2}\partial_\alpha\bar{\psi}_{-1} + \sqrt{6}\psi_0\partial_\alpha\bar{\psi}_1\\
			& + \sqrt{6}\psi_{-1}\partial_\alpha\bar{\psi}_0 - \Omega F_+L_z\alpha + \gamma \Big[ 2\alpha\bar{\psi}_2L_+\psi_2 + \alpha\bar{\psi}_1L_+\psi_1 + \alpha\psi_{-1}L_+\bar{\psi}_{-1} + 2\alpha\psi_{-2}L_+\bar{\psi}_{-2}\\
			& - \left(2|\psi_1|^2 + 3|\psi_0|^2 + 2|\psi_{-1}|^2\right)L_+\alpha - (\sqrt{6}\psi_0\bar{\psi}_2 + 3\psi_{-1}\bar{\psi}_1 + \sqrt{6}\psi_{-2}\bar{\psi}_0)L_-\alpha \Big] \Bigg) \,\mathrm{d}\mathbf{x}.
		\end{align*}
		Hence, we have
		$$\frac{d^2 \delta_r(t)}{dt^2} = -4 \gamma_r^2 \delta_r(t) + 4 \mathcal{E} \left( \Psi(\cdot, 0) \right) + 4 \Omega \langle L_z \rangle(t) + G(\gamma, \Psi),$$
		where
		\begin{align*}
			G(\gamma, \Psi) &:= G_1(x, \gamma, \Psi) + G_1(y, \gamma, \Psi) \\
			&= 2\gamma \int_{\mathbb{R}^d} -\Omega\left(xF_x+yF_y\right) + \gamma \left( 2\rho + 3\left|\psi_1\right|^2 + 4\left|\psi_0\right|^2 + 3\left|\psi_{-1}\right|^2 + \sum_{\ell=-2}^{2}\ell\bar{\psi}_{\ell}L_z\psi_{\ell} \right) \,\mathrm{d}\mathbf{x}.
		\end{align*}
		As a result, we have $G(0, \Psi) = 0$ and $\langle L_z \rangle(t) \equiv \langle L_z \rangle(0)$ when $\gamma = 0$. Consequently, the $\delta_r(t)$ given in \eqref{delta_r} is the unique solution to second order ODE \eqref{d^2delta} with the initial data $\delta_r(0) = \delta_r^{(0)}$ and $\dot{\delta}_r(0) = \delta_r^{(1)}$. Furthermore, if $\Psi^0(\mathbf{x})$ is radial symmetric, the solution $\Psi(\mathbf{x}, t)$ is also radial symmetric since $\gamma_x = \gamma_y$, which implies that
		\[
		\delta_x(t) = \delta_y(t) = \frac{1}{2} \delta_r(t), \quad t \geq 0.
		\]
		This completes the proof.
	\end{prf}
	
\section{Numerical methods}\label{sec-nm}
	In this section, we develop a high-order compact splitting Fourier spectral method,
 where the Hamiltonian is split into a linear part $\mathcal{A}$ (comprising the Laplace, rotation, and SOC terms) and a nonlinear part $\mathcal{B}$ (consisting of all remaining terms). 
 Each subproblem can be integrated exactly in phase or physical space,
  as detailed in Subsection \ref{linearsubproblem} for the linear subproblem  and  in Subsection \ref{nonlinearsubproblem} for the nonlinear one.

	\subsection{Exact integrator for Laplace-Rotation-SOC subproblem}\label{linearsubproblem}
	In this subsection, we propose an exact and efficient integrator to solve the following Laplace-Rotation-SOC subproblem
	\begin{equation}\label{linearEQ1}
		\left\{\begin{aligned}
			&\im\partial_t\Psi(\mathbf{x},t)= \Big[\big(-\frac{1}{2}\Delta-\Omega L_z\big) I_5-\gamma\mathcal{S} \Big] \Psi(\mathbf{x},t) :=\mathcal{A}\Psi(\mathbf{x},t),\quad t_n\le t\le t_{n+1},\\
			&\Psi(\mathbf{x},t_n)=\Psi^n,\quad \mathbf{x}\in \mathbb{R}^d.
		\end{aligned}\right.
	\end{equation}
Inspired by \cite{LY2025}, we observe that,  for the rotated wave function obtained by introducing a function-rotation mapping through a rotation of variables, the rotation term vanishes and the Laplace keeps unchanged, but the SOC term becomes time-dependent. 
This transformation removes the numerical challenge induced by the rotation term, but the time-dependent SOC term makes the resulting variable-coefficient ODE system difficult to calculate exactly.    
To address this issue, we introduce a \textbf{function-rotation mapping with a phase factor} defined as
   \begin{equation}\label{mapping}
   	\phi_{\ell}(\mathbf{x},t):=e^{-\im\ell\Omega t}\psi_{\ell}(\mathcal{R}(t)\mathbf{x},t),\qquad \mathbf{x}\in\mathbb{R}^d,
   \end{equation}
   where $\mathcal{R}(t)$ is a time-dependent rotation matrix, given by    
	\begin{center}
		$\mathcal{R}(t)=$
		$\begin{pmatrix}
			\cos(\Omega t) & \sin(\Omega t)\\
			-\sin(\Omega t) & \cos(\Omega t)
		\end{pmatrix}$, $~$ if $d=2$,$~~~$
		$\mathcal{R}(t)=$
		$\begin{pmatrix}
			\cos(\Omega t) & \sin(\Omega t) & 0\\
			-\sin(\Omega t) & \cos(\Omega t) & 0\\
			0 & 0 & 1
		\end{pmatrix}$, $~$ if $d=3$.
	\end{center}
Using the chain rule, we obtain
    \begin{equation*}
    	\begin{aligned}
    		\partial_t\phi_\ell(\mathbf{x},t) &= -\im\ell\Omega e^{-\im\ell\Omega t} \psi_\ell(\tilde{\mathbf{x}},t) + e^{-\im\ell\Omega t} \big[\partial_t\psi_\ell(\tilde{\mathbf{x}},t) - \im\Omega  L_z\psi_\ell(\tilde{\mathbf{x}},t)\big],~~ 
    		\Delta\psi_\ell(\tilde{\mathbf{x}},t) = e^{\im\ell\Omega t} \Delta\phi_\ell(\mathbf{x},t),\\
    		L_+\psi_\ell(\tilde{\mathbf{x}},t) &= e^{\im\ell\Omega t}\left[\im\left(-\sin(\Omega t)\partial_x + \cos(\Omega t)\partial_y\right) + \left(\cos(\Omega t)\partial_x + \sin(\Omega t)\partial_y\right)\right]\phi_\ell(\mathbf{x},t) = e^{\im(\ell-1)\Omega t}L_{+}\phi_\ell(\mathbf{x},t),\\
    		L_-\psi_\ell(\tilde{\mathbf{x}},t) &= e^{\im\ell\Omega t}\left[\im\left(\cos(\Omega t)\partial_x + \sin(\Omega t)\partial_y\right) + \left(-\sin(\Omega t)\partial_x + \cos(\Omega t)\partial_y\right)\right]\phi_\ell(\mathbf{x},t) = e^{\im(\ell+1)\Omega t}L_{-}\phi_\ell(\mathbf{x},t),
    	\end{aligned}
    \end{equation*}
    where $\tilde{\mathbf{x}} := \mathcal{R}(t)\mathbf{x}$.
    From \eqref{linearEQ1} and these relations, a simple derivation yields the following system for the vector $\Phi := (\phi_2,\phi_1,\phi_0,\phi_{-1},\phi_{-2})^{\top}$.
    \begin{equation}\label{linearEQ2}
    	\left\{\begin{aligned}
    		&\im\partial_t\Phi(\mathbf{x},t)= \Big[\big(-\frac{1}{2}\Delta\big)I_5-\gamma\mathcal{S} +\Omega f_z \Big]\Phi(\mathbf{x},t)  :=\widetilde{\mathcal{A}}\Phi(\mathbf{x},t),\quad t_n\le t\le t_{n+1},\\
    		&\Phi(\mathbf{x},t_n)=\Psi(\mathcal{R}(t_n)\mathbf{x},t_n):=\Phi^n,\quad \mathbf{x}\in \mathbb{R}^d.
    	\end{aligned}\right.
    \end{equation}
    The Laplace and SOC terms remain unchanged, while the rotation term is eliminated and replaced by a constant term $\Omega f_z$. Consequently, the resulting system becomes straightforward to integrate exactly in Fourier space.

    \begin{remark}
     The introduction of a phase factor in the function-rotation mapping eliminates the time dependence of the coefficient matrix $\widetilde{\mathcal{A}}$. 
    In the absence of the phase factor \cite{LY2025}, the SOC term becomes time-dependent, thereby destroying the constant-coefficient structure of the linear soperator. With the phase factor included, the transformed system retains a time-invariant coefficient matrix $\widetilde{\mathcal{A}} = \big(-\frac{1}{2}\Delta\big)I_5-\gamma\mathcal{S} +\Omega f_z$ and thus admits exact integration in Fourier space.
    \end{remark}

    Since the wave function $\psi_{\ell}$ is smooth and decays exponentially fast, its mapped function $\phi_{\ell}$ is also smooth and rapidly decaying.
Therefore, it is reasonable to truncate the whole space $\mathbb{R}^d$ into a common large enough bounded domain  $\mathcal{D}=[-L, L]^d$, $L>0$ and
impose periodic boundary conditions for both $\psi_{\ell}$ and $\phi_{\ell}$. 
     We then apply the Fourier spectral method \cite{S2011} to approximate both wave functions and their spatial derivatives. For simplicity, we take 2D case as an example. The spatial mesh size is chosen as
    $h = {2L}/{N}$ with $N \in 2 \mathbb{Z}^{+}$.
    The Fourier, physical index and grid points sets are defined as follows
    \begin{align}
    	\mathcal{I}_N&=\left\{(j,k)\in\mathbb{Z}^2\mid 0\leq j\leq N - 1,~~ 0\leq k\leq N - 1\right\},\nonumber\\[0.2em]
    	\mathcal{T}_N&=\left\{(p,q)\in\mathbb{Z}^2\mid -N/2\leq p\leq N/2 - 1,~~  -N/2\leq q\leq N/2 - 1\right\},\nonumber\\[0.2em]
    	\mathcal{G}&=\left\{(x_j,y_k):=(-L + jh,-L+kh),\quad(j,k)\in\mathcal{I}_N\right\}.\nonumber
    \end{align}
    In order to discrete \eqref{linearEQ2}, we approximate the function $\phi_{\ell}$ by applying the Fourier spectral method
    \begin{equation*}\label{iFFT}
    	\phi_{\ell}(x,y,t)\approx\sum_{(p,q)\in\mathcal{T}_N}\widehat{\phi}_{\ell,pq}(t)~e^{\im\nu_p(x + L)}e^{\im\nu_q(y + L)},\qquad(x,y)\in\mathcal{D},
    \end{equation*}
    where $\nu_p = 2\pi p/(2L),\nu_q = 2\pi q/(2L)$. The discrete Fourier coefficients are given by
    \begin{equation*}\label{FFT}
\widehat{\phi}_{\ell,pq}(t)\approx\frac{1}{N^2}\sum_{(j,k)\in\mathcal{I}_N}\phi_{\ell,jk}(t)~e^{-\im\nu_p(x_j+L)}e^{-\im\nu_q(y_k+L)},
\qquad(p,q)\in\mathcal{T}_N,
    \end{equation*}
    where $\phi_{\ell,jk}(t)$ is the numerical approximation of $\phi_{\ell}(x_j,y_k,t)$.
    
 Applying the Fourier spectral approximation in space, the semi-discrete of \eqref{linearEQ2} 
 reduces to the following  linear system for the Fourier coefficients
	\begin{equation}\label{linearEQ3}
		\left\{\begin{aligned}
			&\im\partial_t\widehat{\Phi}_{pq}(t)=\widehat{\mathcal{A}}~\widehat{\Phi}_{pq}(t),\quad t_n\le t\le t_{n+1},\\
			&\widehat{\Phi}_{pq}(t_n):=(\widehat{\Phi^n})_{pq},\quad (p,q)\in \mathcal{T}_N
		\end{aligned}\right.
	\end{equation}
with $\widehat{\Phi}_{pq}:=\left(\widehat{\phi}_{2,pq},\widehat{\phi}_{1,pq},\widehat{\phi}_{0,pq},\widehat{\phi}_{-1,pq},\widehat{\phi}_{-2,pq}\right)^{\top}$. The constant coefficient matrix is given by 
	\begin{center}
		$\widehat{\mathcal{A}}=\frac{1}{2}(\nu_p^2+\nu_q^2)I_{5}-$
		$\begin{pmatrix}
			-2\Omega & \gamma a & 0 & 0 & 0\\
			\gamma\bar{a} & -\Omega & \frac{\sqrt{6}}{2}\gamma a & 0 & 0\\
			0 & \frac{\sqrt{6}}{2}\gamma\bar{a} & 0 & \frac{\sqrt{6}}{2}\gamma a & 0\\
			0 & 0 & \frac{\sqrt{6}}{2}\gamma\bar{a} & \Omega & \gamma a\\
			0 & 0 & 0 & \gamma\bar{a} & 2\Omega
		\end{pmatrix}$
		$:=\frac{1}{2}(\nu_p^2+\nu_q^2)I_{5}-Q$
	\end{center}
	with $a:=-\nu_q-\im \nu_p$. 
The equation \eqref{linearEQ3} is actually an
ODE system with constant skew-Hermitian coefficient matrix, then we have
    \begin{equation*}\label{sl3}
    	\widehat{\Phi}_{pq}(t)=e^{-\im(t-t_n)\widehat{\mathcal{A}}}~(\widehat{\Phi^n})_{pq}=
    e^{-\frac{1}{2}\im(\nu_p^2+\nu_q^2)(t-t_n)}e^{\im(t-t_n)Q}~(\widehat{\Phi^n})_{pq}.
    \end{equation*}
    Therefore, the solution  to system \eqref{linearEQ2} at $t = t_{n+1}$ can be written as
    \begin{equation}\label{sl1}
    	\Phi_{jk}^{n+1} = \sum_{(p,q)\in\mathcal{T}_N} e^{-\frac{1}{2}\im(\nu_p^2+\nu_q^2)\tau} e^{\im \tau Q}~ (\widehat{\Phi^n})_{pq}~e^{\im\nu_p(x_j+L)}e^{\im\nu_q(y_k+L)}.
    \end{equation}
    Specifically,  the matrix exponential $e^{\im \tau Q}$ can be computed explicitly as
    \begin{center}
    	$e^{\im \tau Q}=$
    	$\begin{pmatrix}
    		c_{11} & ac_{12} & a^2c_{13} & a^3c_{14} & a^4c_{15}\\
    		\bar{a}c_{12} & c_{22} & ac_{23} & a^2c_{24} & a^3c_{25}\\
    		\bar{a}^2c_{13} & \bar{a}c_{23} & c_{33} & ac_{34} & a^2c_{35}\\
    		\bar{a}^3c_{14} & \bar{a}^2c_{24} & \bar{a}c_{34} & c_{44} &ac_{45}\\
    		\bar{a}^4c_{15} & \bar{a}^3c_{25} & \bar{a}^2c_{35} & \bar{a}c_{45} & c_{55}
    	\end{pmatrix}$,
   	\end{center}
    where
    {\allowdisplaybreaks\begin{align*}
    	c_{11}=&\frac{1}{8\lambda^4}\left[(8\Omega^4+8\Omega^2\xi^2+\xi^4)\eta_1+4\xi^2(2\Omega^2+\xi^2)\eta_2-4\Omega\lambda(2\Omega^2+\xi^2)\eta_3-8\Omega\lambda \xi^2\eta_4+8\lambda^4\right],\\
    	c_{12}=&\frac{\gamma }{4\lambda^4}\left[-\Omega(4\Omega^2+3\xi^2)\eta_1+4\Omega^3\eta_2+\lambda(4\Omega^2+\xi^2)\eta_3-2\lambda(2\Omega^2-\xi^2)\eta_4\right],\\
    	c_{13}=&\frac{\sqrt{6}\gamma^2}{4\lambda^4}\left[(2\Omega^2+\xi^2)\eta_2-2\Omega\lambda\eta_4+2\lambda^2\right]\eta_2,\quad c_{14}=\frac{\gamma^3}{2\lambda^4}\left(-\Omega\eta_2+\lambda\eta_4\right)\eta_2,\quad c_{15}=\frac{\gamma^4}{8\lambda^4}\left(\eta_1-4\eta_2\right),\\
    	c_{22}=&\frac{1}{2\lambda^4}\left(2\xi^2\eta_2+\lambda^2\right)\left[(2\Omega^2+\xi^2)\eta_2-2\Omega\lambda\eta_4+2\lambda^2\right],\quad c_{23}=\frac{\sqrt{6}\gamma }{2\lambda^4}\left(\xi^2\eta_2+\lambda^2\right)\left(-\Omega\eta_2+\lambda\eta_4\right),\\
    	c_{24}=&\frac{\gamma^2}{2\lambda^4}\left(2\xi^2\eta_2+3\lambda^2\right)\eta_2,\quad c_{25}=c_{14}+\frac{\Omega\gamma^3}{\lambda^4}\eta_2^2,\quad c_{33}=\frac{3\xi^2}{4\lambda^4}(\xi^2\eta_1+4\Omega^2\eta_2)+1,\\
    	c_{34}=&c_{23}+\frac{\sqrt{6}\Omega\gamma }{\lambda^4}\left(\xi^2\eta_2+\lambda^2\right)\eta_2,\quad c_{35}=c_{13}+\frac{\sqrt{6}\Omega\gamma^2}{\lambda^3}\eta_2\eta_4,\quad c_{44}=c_{22}+\frac{2\Omega}{\lambda^3}(2\xi^2\eta_2+\lambda^2)\eta_4,\\
    	c_{45}=&c_{12}+\frac{\Omega\gamma }{2\lambda^4}\left[(4\Omega^2+3\xi^2)\eta_1-4\Omega^2\eta_2\right],\quad c_{55}=c_{11}+\frac{\Omega}{\lambda^3}\left[(2\Omega^2+\xi^2)\eta_3+2\xi^2\eta_4\right]
    \end{align*}}
    with $\xi=\gamma\left|a\right|,~\lambda=\sqrt{\Omega^2+\xi^2},~\eta_1=\cos{(2\lambda\tau)}-1,~\eta_2=\cos{(\lambda\tau)}-1,
    ~\eta_3=\im\sin{(2\lambda\tau)},~\eta_4=\im\sin{(\lambda\tau)}$.
\vspace{0.3em}
\begin{remark}\label{comput-cost}
In fact, the matrix $e^{\im\tau Q}$ depends only on $\Omega$, $\gamma$ and the time step $\tau$, hence it can be computed once and treated as a pre-computation.
This significantly reduces the computational cost compared with \cite{LY2025}, where the corresponding matrix needs to be updated at each time step.
\end{remark}
\vspace{0.3em}
Finally, we obtain the solution to the linear subproblem \eqref{linearEQ1} as
    \begin{equation*}
    	\psi^{n+1}_{\ell}(\mathbf{x}_{jk})=e^{\im\ell\Omega t}~\phi^{n+1}_{\ell}(\mathcal{R}^{-1}(t_{n+1})\mathbf{x}_{jk}) ~~\text{with}~~ \mathbf{x}_{jk}=(x_j,y_k)^{\top}.
    \end{equation*}

    To ensure that the proposed procedure is computationally viable, an efficient implementation of 
    the function-rotation mapping with a phase factor is crucial. Fortunately,
the inclusion of the phase factor does not introduce any additional computational difficulty, and the mapping can still be efficiently realized using the Rotation-Shear-Decomposition-Acceleration (RSDA) method \cite{LY2025}. Specifically, \begin{align}
    		\phi_{\ell}^n(\mathbf{x}) &= e^{-\im\ell\Omega t_n}e^{a_1 y \partial_x} e^{b_1 x \partial_y} e^{a_1 y \partial_x} \psi_{\ell}^n(\mathbf{x}), \qquad \mathbf{x}\in\mathbb{R}^d,\label{RSDA1}\\[0.5em]
    		\psi_{\ell}^{n+1}(\mathbf{x}) &= e^{\im\ell\Omega t_{n+1}}e^{a_2 y \partial_x} e^{b_2 x \partial_y} e^{a_2 y \partial_x} \phi_{\ell}^{n+1}(\mathbf{x}),\qquad \mathbf{x}\in\mathbb{R}^d,\label{RSDA2}
    	\end{align}
    	where $a_1 = \tan(\Omega t_n/2)$, $b_1 = -\sin(\Omega t_n)$, $a_2 = -\tan(\Omega t_{n+1}/2)$ and $ b_2 = \sin(\Omega t_{n+1})$. The numerical implementation of \eqref{RSDA1} and \eqref{RSDA2}  requires only one-dimensional FFTs and iFFTs,  thereby achieving nearly optimal efficiency. Details can be found in \cite{LY2025}.

    \begin{remark}[Special cases]
    The proposed method naturally resolves special physical cases, such as the vanishing spin-orbit coupling case $(\gamma=0)$ and the non-rotating case $(\Omega=0)$.
    	\begin{itemize}
    		\item \textbf{Non-SOC}: When the spin-orbit coupling strength $\gamma$ goes to zero,
    we have 
    \bea
    \label{NonSOCMat}
    e^{\im \tau Q } \longrightarrow 
    \begin{pmatrix}
			e^{-2 \im \tau \Omega} & 0 & 0 & 0 & 0\\
			0 & e^{- \im \tau \Omega}  & 0 & 0 & 0\\
			0 & 0 & 1 & 0 & 0\\
			0 & 0 & 0 & e^{ \im \tau \Omega} & 0\\
			0 & 0 & 0 & 0 & e^{2 \im \tau \Omega}
		\end{pmatrix}
    \eea
    and the numerical scheme \eqref{sl1} reads as 
   \beas
   \phi_{\ell, jk}^{n+1} = \sum_{(p,q)\in\mathcal{T}_N} e^{-\frac{1}{2}\im(\nu_p^2+\nu_q^2)\tau} e^{-\ell  \im \tau \Omega}~ (\widehat{\phi^n_{\ell}})_{pq}~e^{\im\nu_p(x_j+L)}e^{\im\nu_q(y_k+L)}.
   \eeas
     This scheme aligns with the numerical scheme for solving the PDEs \eqref{linearEQ1} with $\gamma = 0$ using the function-rotation mapping proposed in \cite{LY2025}.
    		\item \textbf{Non-rotating}:
      When the rotating speed $\Omega$ goes to zero, simple calculations imply that
      $\phi_{\ell} = \psi_{\ell}$ and $\widetilde{\mathcal{A}} = \left(-\frac{1}{2}\Delta\right) I_5 - \gamma\mathcal{S} = \mathcal{A}$.  Therefore, the corresponding numerical scheme aligns with the scheme for solving the PDEs \eqref{linearEQ1} with $\Omega = 0$. 
    	\end{itemize}
    \end{remark}
    \begin{remark}[Extension to three-dimensional problem]
    	It is straightforward to extend the above method to three dimensional case,
    	because both the rotation and SOC terms are independent of space variable $z$,
    	and we choose to omit details for brevity.
    \end{remark}

	\subsection{Exact evaluation for nonlinear subproblem}\label{nonlinearsubproblem}
    In this subsection, we present an exact solution to the following nonlinear subproblem \cite{S2017}
	\begin{equation}\label{nonlEQ1}
		\left\{\begin{aligned}
			&\im\partial_t\Psi(\mathbf{x},t)=\big[(V(\mathbf{x})+c_0\rho)I_5+c_1\mathbf{F}\cdot\mathbf{f}\big]\Psi(\mathbf{x},t)+c_2A_{00}A\bar{\Psi}(\mathbf{x},t),\quad t_n\le t\le t_{n+1},\\
			&\Psi(\mathbf{x},t_n)=\Psi^n,\quad \mathbf{x}\in \mathbb{R}^d,
		\end{aligned}\right.
	\end{equation}
	where the nonlinearity arises only from the density $\rho = \Psi^{\mathsf{H}} \Psi$, spin vector $\mathbf{F} = (F_x, F_y, F_z)^{\top}$  and $A_{00}=\Psi^{\top}A\Psi$. In fact, the above equation reduces to a linear ODE, since $\rho$ and  $\mathbf{F}$  are time-independent and $A_{00}(t)$ can be solved analytically.

	Specifically, using the facts that $f_{\alpha}(\alpha=x,y,z)$ and $A$ are Hermitian matrices, and that the commutator relations $\left[f_x,f_y\right]:=f_xf_y-f_yf_x=\im f_z$, $\left[f_y,f_z\right]=\im f_x$, $\left[f_z,f_x\right]=\im f_y$ and $\Psi^{\top}Af_{\alpha}\Psi=0$ hold, we have
	\begin{align}
		\partial_t\rho &= \partial_t\left(\Psi^{\mathsf{H}}\Psi\right)=\im c_1\Psi^{\mathsf{H}}\left((\mathbf{F}\cdot\mathbf{f})^{\mathsf{H}}-\mathbf{F}\cdot\mathbf{f}\right)\Psi+\im c_2\left(\custombar{A_{00}}\Psi^{\top}A^{\mathsf{H}}\Psi-A_{00}\Psi^{\mathsf{H}}A\bar{\Psi}\right)=0,\\[0.4em]
		\partial_tF_{\alpha}&=\partial_t\left(\Psi^{\mathsf{H}}f_{\alpha}\Psi\right)=\im c_1 \Psi^{\mathsf{H}}\left[\mathbf{F}\cdot\mathbf{f},f_{\alpha}\right]\Psi + \im c_2 \left(\custombar{A_{00}}\Psi^{\top}Af_{\alpha}\Psi-A_{00}\Psi^{\mathsf{H}}f_{\alpha}A\bar{\Psi}\right)=0. \label{F_alpha}
	\end{align}
	This implies that $\rho$ and $F_{\alpha}$ are time-invariant, i.e., $\rho(\mathbf{x},t)\equiv\rho(\mathbf{x},t_n):=\rho^n$ and $F_{\alpha}(\Psi)\equiv F_{\alpha}(\Psi^n):=F_{\alpha}^n$ for any time $t_n\leq t\leq t_{n+1}$. For the quantity $A_{00}$, using the identity
	$
	(\mathbf{F}^n\cdot\bar{\mathbf{f}})A+A(\mathbf{F}^n\cdot\mathbf{f})=0$ with
	$\mathbf{F}^n:=\left(F_x^n,F_y^n,F_z^n\right)^{\top}
	$, a simple calculation
	shows that
	\begin{equation*}
			\partial_t {A_{00}}
			= -2\im\left[V + \left(c_0+\frac{c_2}{5}\right) {\rho^n}\right]{A_{00}},
	\end{equation*}
	from which we derive exact and explicit formula for the $A_{00}$ as follows
	\begin{equation*}
		A_{00}(t)=e^{-2\im(t-t_n)\big[V + \left(c_0+\frac{c_2}{5}\right) {\rho^n}\big]}A_{00}^n \quad \mbox{with} \quad A_{00}^n=(\Psi^n)^{\top}A\Psi^n.
	\end{equation*}

Introducing the transformation
    \begin{equation*}
    	\begin{aligned}
    		\tilde{\Psi}(\textbf{x},t) :=& e^{\im (t-t_n)\big( \big[V + (c_0 +\frac{c_2}{5}) {\rho}^n \big]I_5 +c_1\mathbf{F}^n\cdot\mathbf{f}\big)} \Psi(\textbf{x},t),\\
    	\end{aligned}
    \end{equation*}
   and  plugging it into \eqref{nonlEQ1}, we obtain
    \begin{equation}\label{nonlEQ3}
    	\partial_t \tilde{\Psi}(\textbf{x},t) = \im c_2 \Big(\frac{{\rho^n}}{5} \tilde{\Psi} - A_{00}^nA\bar{\tilde{\Psi}}\Big)\quad \mbox{with} \quad \tilde{\Psi}(\mathbf{x},t_n)=\Psi^n.
    \end{equation}
    To solve the ODE \eqref{nonlEQ3}, we take its conjugate and introduce the vector $\zeta(\textbf{x},t) := \big( \tilde{\Psi}^{\top},\bar{\tilde{\Psi}}^{\top} \big)^{\top}$, which satisfies
\begin{equation}\label{nonlEQ4}
		\partial_t \zeta(\textbf{x},t) = \im c_2 \begin{pmatrix}
			\frac{\rho^n}{5} I_5 & -{A_{00}^n}A\\[0.5em]
			\custombar{A_{00}^n}A & -\frac{\rho^n}{5} I_5
		\end{pmatrix} \zeta(\textbf{x},t)
		:= \im c_2 D  \zeta(\textbf{x},t).
\end{equation}
Clearly, the coefficient matrix $D$ is time-independent  and satisfies
$D^2 = \big[\left(\rho^n/5\right)^2 - |{A_{00}^n}|^2 /5\big]I_{10} := S^2I_{10}$.  It follows that the exact solution to ODE \eqref{nonlEQ4} reads as
 \begin{equation*}
	\begin{aligned}
		\zeta(\textbf{x},t) &= e^{\im c_2(t-t_n)D}\zeta(\textbf{x},t_n)
		= \Big[\cos(c_2S(t-t_n))I_{10}+\frac{\im}{S}\sin(c_2S(t-t_n))D\Big]\zeta(\textbf{x},t_n),
	\end{aligned}
\end{equation*}
where $S=\sqrt{\left({{\rho}^n}/5\right)^2 - |{A_{00}^n}|^2 /5}$.
Then we have
    \begin{equation*}
    	\tilde{\Psi}(\textbf{x},t) = \cos\left(c_2S(t-t_n)\right){\Psi}^n + \frac{\im}{S} \sin\left(c_2S(t-t_n)\right)\Big(\frac{{\rho}^n}{5} \Psi^n - {A_{00}^n}A\bar{{\Psi}}^n \Big).
    \end{equation*}
    Therefore, the exact solution to \eqref{nonlEQ1} reads as follows
    \begin{equation}\label{solution2}
    	\begin{split}
    		\Psi(\textbf{x},t) &= e^{-\im (t-t_n)\big[V + (c_0 +c_2/5) {\rho}^n \big] } e^{-\im c_1 (t-t_n) {\mathbf{F}}^n\cdot \mathbf{f} }\\ &\left[\cos(c_2S(t-t_n)){\Psi}^n + \frac{\im}{S} \sin(c_2S(t-t_n))\Big( \frac{{\rho}^n}{5} \Psi^n - {A_{00}}^nA\bar{\Psi}^n\Big)\right],
    	\end{split}
    \end{equation}
    where $e^{-\im c_1(t-t_n)\mathbf{F}^n\cdot\mathbf{f}}$ is calculated as
    \begin{equation*}
    	\begin{aligned}
    		e^{-\im c_1 (t-t_n)\mathbf{F}^n\cdot\mathbf{f}}=&I_5 + \im\left(\frac{1}{6}\sin(2\kappa)-\frac{4}{3}\sin(\kappa)\right)\frac{\mathbf{F}^n\cdot\mathbf{f}}{|\mathbf{F}^n|}
    		+\left(\frac{4}{3}\cos(\kappa)-\frac{1}{12}\cos(2\kappa)-\frac{5}{4}\right)\frac{(\mathbf{F}^n\cdot\mathbf{f})^2}{|\mathbf{F}^n|^2}\\
    		&+\im\left(\frac{1}{3}\sin(\kappa)-\frac{1}{6}\sin(2\kappa)\right)\frac{(\mathbf{F}^n\cdot\mathbf{f})^3}{|\mathbf{F}^n|^3}
    		+\left(\frac{1}{12}\cos(2\kappa)-\frac{1}{3}\cos(\kappa)+\frac{1}{4}\right)\frac{(\mathbf{F}^n\cdot\mathbf{f})^4}{|\mathbf{F}^n|^4}
    	\end{aligned}
    \end{equation*}
    with $\kappa=c_1|\mathbf{F}^n|(t-t_n)$.

\subsection{High-order compact splitting Fourier spectral method}\label{sec4.3}
In this subsection, we construct high-order time marching schemes
based on operator composition techniques. 
Let
$
\Psi(t)=e^{-\im(t-t_n)\mathcal{A}}\Psi^n$ and $\Psi(t)=e^{-\im(t-t_n)\mathcal{B}}\Psi^n
$
denotes the solutions to linear subproblem and nonlinear subproblem respectively. In principle, 
high-order splitting schemes can be constructed as \cite{Y1990}
\[
\Psi^{n + 1}=\left(\prod_{j = 1}^{m}e^{-\im a_j\tau\mathcal{A}}e^{-\im b_j\tau\mathcal{B}}\right)\Psi^n,
\]
where the coefficients $a_j$ and $b_j~(j=1,\cdots,m)$ are chosen properly.  Here, we present two commonly used splitting schemes.
\begin{itemize}
	\item \textbf{Second-order Strang splitting}: $m=2$, $a_1=a_2=1/2$, $b_1=1$, $b_2=0$,
	\item \textbf{Fourth-order symplectic time integrator}: $m=4$, $a_1=a_4=\frac{1}{2(2-2^{1/3})}$, $a_2=a_3=\frac{1-2^{1/3}}{2(2-2^{1/3})}$, $b_1=b_3=\frac{1}{2-2^{1/3}}$, $b_2=-\frac{2^{1/3}}{2-2^{1/3}}$, $b_4=0$.
\end{itemize}

In practice, from time $t=t_n$ to $t=t_{n+1}$, we combine the splitting steps via the standard Strang splitting and present detailed step-by-step algorithm in Algorithm \ref{Alg2}.
\begin{algorithm}
	\caption{Second-order compact splitting Fourier spectral method}
	\label{Alg2}
	\KwIn{Initial data $\Psi^{n}$ and time step $\tau$.}
	\begin{algorithmic}[1]
		\State Solve the nonlinear subproblem by \eqref{solution2} for half time step $\tau/2$ with initial data $\Psi^n$ to obtain $\Psi^{\ast} $.
		\State Solve the linear subproblem for time step $\tau$ with the data $\Psi^{\ast}=\left(\psi_2^{\ast},\psi_1^{\ast},\psi_0^{\ast},\psi_{-1}^{\ast},\psi_{-2}^{\ast}\right)^{\top}$:
		\begin{itemize}
			\item Compute the mapping  $\phi_{\ell}^n=e^{-\im\ell\Omega t}\psi^{\ast}_{\ell}(\mathcal{R}(t_n)\mathbf{x})$ by \eqref{RSDA1}. 
		    \item Solve the ODEs \eqref{linearEQ2} to obtain $\phi^{n+1}_{\ell}$ using \eqref{sl1}.
		    \item Compute the mapping $\psi^{\ast\ast}_{\ell}=e^{\im\ell\Omega t}\phi^{n+1}_{\ell}(\mathcal{R}^{-1}(t_{n+1})\mathbf{x})$ by \eqref{RSDA2}.
		\end{itemize}
		\State Solve the nonlinear subproblem by \eqref{solution2} for half time step $\tau/2$ with the data $\Psi^{\ast\ast}=\left(\psi_2^{\ast\ast},\cdots,\psi_{-2}^{\ast\ast}\right)^{\top}$ to obtain the final solution $\Psi^{n+1}$.
	\end{algorithmic}
	\KwOut{Numerical solution $\Psi^{n+1}$.}
\end{algorithm}

\begin{remark}[Efficiency]\label{NlogN}
	In Algorithm 1, we can reduce the number of Fourier-Physical switches by merging adjacent computation. As a result, only  $30N$ pairs of one-dimensional FFTs and iFFTs are required in step 2, and the complexity is $O(N^2 \log N)$.  For the 3D case, the schemes requires $35 N^2$ pairs of one-dimensional FFTs and iFFTs, and the complexity becomes $O(N^3 \log N)$.
\end{remark}

\subsection{Stability}
In this subsection, we prove that the proposed method preserves the conservation of both mass and magnetization at the discrete level. Without loss of generality, we present proofs in the 2D case and extension to 3D is straightforward.
We define the discrete $l^2$-norm of $\psi_{\ell}^n$ as
$ \| \psi_{\ell}^n \|_{l^2} = (h^2 \sum_{j=0}^{N-1} \sum_{k=0}^{N-1}|\psi^n_{\ell,jk}|^2)^{\frac{1}{2}}$,
and define $\|\Psi^n\|_{l^2}:=\left(\sum_{\ell=-2}^{2}\|\psi_{\ell}^n\|_{l^2}^2\right)^{\frac{1}{2}}$.

    \begin{lem}[Stability]
    	The mass $\mathcal{N}(\Psi)$ is conserved at discrete level. In fact, for any $h$, $\tau>0$, we have
    	$$\|\Psi^n\|_{l^2}^2=\|\Psi^0|_{l^2}^2.$$
    	This reveals that the compact splitting Fourier spectral method is unconditionally stable.
    \end{lem}
    \begin{prf}
        For the nonlinear subproblem \eqref{nonlEQ1}, the exact solution \eqref{solution2} can be rewritten as 
        $$\Psi^{n+1}=e^{-\im\tau \mathcal{B}_1^n}\mathbf{b}^n, $$
        where $\mathcal{B}_1^n:=(V+(c_0+c_2/5)\rho^n)I_5+c_1\mathbf{F}^n\cdot\mathbf{f}$ and $\mathbf{b}^n:=\cos(c_2S\tau){\Psi}^n + \frac{\im}{S} \sin(c_2S\tau)\big( \frac{{\rho}^n}{5} \Psi^n - {A_{00}}^nA\bar{\Psi}^n\big)$. Using $(\mathcal{B}_1^n)^{\mathsf{H}}=\mathcal{B}_1^n$, a simple calculation shows that 
        \begin{equation*}
        	\|\Psi^{n+1}\|_{l^2}^2	=h^2\sum_{j=0}^{N-1}\sum_{k=0}^{N-1}\left(\left(\mathbf{b}^n\right)^{\mathsf{H}}e^{\im(\mathcal{B}_1^n)^{\mathsf{H}}\tau} e^{-\im\mathcal{B}_1^n\tau}\mathbf{b}^n\right)_{jk}
        	=h^2\sum_{j=0}^{N-1}\sum_{k=0}^{N-1}\left(\rho^n\right)_{jk}=\|\Psi^{n}\|_{l^2}^2,
        \end{equation*}
        where $(f^n)_{jk}:=f(x_j,y_k,t_n)$.
While for the linear subproblem, we only need to consider \eqref{RSDA1}, \eqref{RSDA2} and \eqref{sl1}. Based on Parseval's identity and following the methodology in \cite{LY2025}, we have 
\bea
\label{MassConLine}
\|\Psi^{n+1}\|^2_{l^2}=\|\Phi^{n+1}\|^2_{l^2}=\|\Phi^{n}\|^2_{l^2}=\|\Psi^{n}\|^2_{l^2}.
\eea
        This completes the proof.

    \end{prf}

    \begin{lem}[Magnetization conservation]
    	The magnetization $\mathcal{M}(\Psi)$ is conserved at discrete level when $\gamma=0$. In fact, for any $h$, $\tau>0$, we have
    $$\sum_{\ell=-2}^{2}\ell\|\psi_{\ell}^n\|_{l^2}^2=\sum_{\ell=-2}^{2}\ell\|\psi_{\ell}^0\|_{l^2}^2.$$
    \end{lem}
    \begin{prf}
    	For the linear subproblem \eqref{linearEQ1}, we know that $e^{\im \tau Q}$ is a diagonal unitary matrix when $\gamma=0$, as shown in \eqref{NonSOCMat}. Therefore, it follows from \eqref{MassConLine} that
    	$$\big\|\psi_{\ell}^{n+1}\big\|_{l^2}^2=\big\|\psi_{\ell}^{n}\big\|_{l^2}^2.$$
    	For the nonlinear subproblem \eqref{nonlEQ1}, with the help of \eqref{F_alpha}, we have
    	\begin{equation*}
    		\partial_t\left(\sum_{\ell=-2}^{2}\ell\big\|\psi_{\ell}\big\|_{l^2}^2\right)=h^2\sum_{j=0}^{N-1}\sum_{k=0}^{N-1}\left(\partial_tF_z\right)_{jk}=0.
    	\end{equation*}
    This implies that
$$\sum_{\ell=-2}^{2}\ell\big\|\psi_{\ell}^{n+1}\big\|_{l^2}^2=\sum_{\ell=-2}^{2}\ell\big\|\psi^n_{\ell}\big\|_{l^2}^2.$$
    	The proof is completed.
    \end{prf}

\section{Numerical results}\label{sec-nr}
    In this section, we begin by evaluating the performance of our numerical method through the temporal/spatial accuracies and efficiency test. We then apply it to characterize the dynamical laws, including the conservation of mass, energy, magnetization, along with the evolution of angular momentum expectation and condensate widths. Finally, we investigate some interesting phenomena, such as the effect of SOC on dynamics and the dynamics of the vortex lattice. In the following simulations, the external potential $V(\mathbf{x})$ is chosen as \eqref{harmonic} with $\gamma_x=\gamma_y=1$ unless stated otherwise. For convenience, throughout this section we denote the second-order/fourth-order Fourier spectral method as \textbf{TS2}/\textbf{TS4}.
\subsection{Accuracy confirmation.}

    In this section, we define the numerical errors as
    $$E_{\ell}^{h,\tau}=\big\|\psi_{\ell}^{\mathrm{ref}}-\psi_{\ell}^{h,\tau}\big\|_{l^2}/\big\|\psi_{\ell}^{\mathrm{ref}}\big\|_{l^2},\quad \ell=-2,-1,0,1,2,$$
    where $\psi_{\ell}^{\mathrm{ref}}$ is a high-accuracy reference solution, and $\psi_{\ell}^{h,\tau}$ is a numerical solution obtained with the time step $\tau$ and the spatial mesh size $h$.

    To confirm the temporal convergence, we compute the wave function with a small mesh size $h_0=2^{d-7}$, and the reference solution $\psi_{\ell}^{\mathrm{ref}}$ is obtained by TS4 with mesh size $h_0$ and a small time step $\tau_0=10^{-4}$. To confirm the spatial convergence, we compute the wave function with different mesh size $h$ and the reference solution is obtained by TS2/TS4 with $h_0$ and $\tau_0$.


    \begin{exmp}[\textbf{\textit{Accuracy}}]\label{accuracy}
        In this example, we test the temporal and spatial accuracy in both 2D and 3D cases. To this end, We consider the following two cases
    	\begin{itemize}
    		\item $\mathrm{\mathbf{2D~case:}}$ $c_0=100$, $c_1=-1$, $c_2=1$, $\Omega=0.2$, $\gamma=0.3$.
    		\item $\mathrm{\mathbf{3D~case:}}$ $c_0=10$, $c_1=-1$, $c_2=1$, $\Omega=0.2$, $\gamma=0.1$.
    	\end{itemize}
        \noindent The initial functions are
    	\begin{equation}\label{ini1}
    		\psi_{\pm2}^0(\mathbf{x})=\phi(\mathbf{x}),\quad\psi_{\pm1}^0(\mathbf{x})=\phi(\mathbf{x}),\quad\psi_0^0(\mathbf{x})=2\phi(\mathbf{x}),
    	\end{equation}
    	where $\phi(\mathbf{x})=e^{-|\mathbf{x}|^2/2}/\left(\sqrt{8}\pi^{d/4}\right)$.
    \end{exmp}

    \begin{table}[h!]
    	\centering
    	\caption{\footnotesize Numerical errors of TS2 and TS4 at time $t=0.5$ for 2D case in \textbf{Example \ref{accuracy}}.}
    	\label{Acc1}
    	{\small\begin{tabular}{ccccccc}
    		\toprule
    		\multirow{19}{*}{\raisebox{-12ex}{\begin{tabular}[c]{@{}c@{}}Temporal \\ direction\end{tabular}}}
    		&  & $\tau$ & 1/80 & 1/160 & 1/320 & 1/640 \\ \cmidrule{2-7}
    		& \multirow{9}{*}{\raisebox{-3ex}{TS2}} & $E_2^{h_0,\tau}$ & 8.4374E-04 & 2.1064E-04 & 5.2641E-05 & 1.3159E-05 \\
    		&                      & rate & & 2.0020 & 2.0005 & 2.0002 \\
    		&                      & $E_1^{h_0,\tau}$ & 8.4254E-04 & 2.1042E-04 & 5.2592E-05 & 1.3147E-05 \\
    		&                      & rate & & 2.0015 & 2.0004 & 2.0001 \\
    		&                      & $E_0^{h_0,\tau}$ & 5.4795E-04 & 1.3690E-04 & 3.4222E-05 & 8.5552E-06 \\
    		&                      & rate & & 2.0008 & 2.0002 & 2.0001 \\
    		&                      & $E_{-1}^{h_0,\tau}$ & 8.3759E-04 & 2.0918E-04 & 5.2282E-05 & 1.3069E-05 \\
    		&                      & rate & & 2.0015 & 2.0004 & 2.0001 \\
    		&                      & $E_{-2}^{h_0,\tau}$ & 8.3704E-04 & 2.0896E-04 & 5.2224E-05 & 1.3054E-05 \\
    		&                      & rate & & 2.0020 & 2.0005 & 2.0001 \\ \cmidrule{2-7}
    		& \multirow{9}{*}{\raisebox{-3ex}{TS4}} & $E_2^{h_0,\tau}$ & 4.1380E-05 & 2.6390E-06 & 1.6583E-07 & 1.0379E-08 \\
    		&                      & rate & & 3.9708 & 3.9922 & 3.9980 \\
    		&                      & $E_1^{h_0,\tau}$ & 3.4689E-05 & 2.2035E-06 & 1.3833E-07 & 8.6550E-09 \\
    		&                      & rate & & 3.9766 & 3.9937 & 3.9984 \\
    		&                      & $E_0^{h_0,\tau}$ & 2.1960E-05 & 1.3992E-06 & 8.7900E-08 & 5.2021E-09 \\
    		&                      & rate & & 3.9722 & 3.9925 & 3.9978 \\
    		&                      & $E_{-1}^{h_0,\tau}$ & 3.4512E-05 & 2.1923E-06 & 1.3762E-07 & 8.6109E-09 \\
    		&                      & rate & & 3.9766 & 3.9937 & 3.9984 \\
    		&                      & $E_{-2}^{h_0,\tau}$ & 4.0945E-05 & 2.6112E-06 & 1.6408E-07 & 1.0269E-08 \\
    		&                      & rate & & 3.9709 & 3.9922 & 3.9980 \\
    		\toprule
    		\multirow{11}{*}{\raisebox{-12ex}{\begin{tabular}[c]{@{}c@{}}Spatial \\ direction\end{tabular}}}
    		& & $h$ & 1/2  & 1/4  & 1/8   & 1/16  \\ \cmidrule{2-7}
    		& \multirow{5}{*}{\raisebox{-2.5ex}{TS2}} & $E_2^{h,\tau_0}$ & 2.1323E-01 & 3.2161E-03 & 4.5656E-08 & 1.3003E-12 \\ [2.5pt]
    		&                      & $E_1^{h,\tau_0}$ & 2.1990E-01 & 3.4138E-03 & 3.8468E-08 & 1.2759E-12 \\ [2.5pt]
    		&                      & $E_0^{h,\tau_0}$ & 1.7801E-01 & 2.3852E-03 & 3.5863E-08 & 1.3317E-12 \\ [2.5pt]
    		&                      & $E_{-1}^{h,\tau_0}$ & 2.2702E-01 & 3.5795E-03 & 4.2493E-08 & 1.2192E-12 \\ [2.5pt]
    		&                      & $E_{-2}^{h,\tau_0}$ & 2.1287E-01 & 3.4759E-03 & 5.0031E-08 & 1.2682E-12 \\ \cmidrule{2-7}
    		& \multirow{5}{*}{\raisebox{-2.5ex}{TS4}} & $E_2^{h,\tau_0}$ & 2.1323E-01 & 3.2161E-03 & 4.5658E-08 & 3.7094E-12 \\ [2.5pt]
    		&                      & $E_1^{h,\tau_0}$ & 2.1990E-01 & 3.4138E-03 & 3.8470E-08 & 3.5924E-12 \\ [2.5pt]
    		&                      & $E_0^{h,\tau_0}$ & 1.7801E-01 & 2.3852E-03 & 3.5864E-08 & 3.6499E-12 \\ [2.5pt]
    		&                      & $E_{-1}^{h,\tau_0}$ & 2.2702E-01 & 3.5795E-03 & 4.2494E-08 & 3.3506E-12 \\ [2.5pt]
    		&                      & $E_{-2}^{h,\tau_0}$ & 2.1287E-01 & 3.4759E-03 & 5.0032E-08 & 3.4201E-12 \\
    		\bottomrule
    	\end{tabular}}
    \end{table}

    \begin{table}[h!]
    	\centering
    	\caption{\footnotesize Numerical errors of TS2 and TS4 at time $t=0.3$ for 3D case in \textbf{Example \ref{accuracy}}.}
    	\label{Acc2}
    	{\small\begin{tabular}{ccccccc}
    			\toprule
    			\multirow{19}{*}{\raisebox{-12ex}{\begin{tabular}[c]{@{}c@{}}Temporal \\ direction\end{tabular}}}
    			&  & $\tau$ & 1/40 & 1/80 & 1/160 & 1/320 \\ \cmidrule{2-7}
    			& \multirow{9}{*}{\raisebox{-3ex}{TS2}} & $E_2^{h_0,\tau}$ & 7.3666E-05 & 1.8405E-05 & 4.6005E-06 & 1.1501E-06 \\
    			&                      & rate & $*$ & 2.0009 & 2.0002 & 2.0001 \\
    			&                      & $E_1^{h_0,\tau}$ & 4.5933E-05 & 1.1477E-05 & 2.8688E-06 & 7.1718E-07 \\
    			&                      & rate & $*$ & 2.0008 & 2.0002 & 2.0001 \\
    			&                      & $E_0^{h_0,\tau}$ & 6.9868E-05 & 1.7456E-05 & 4.3633E-06 & 1.0908E-06 \\
    			&                      & rate & $*$ & 2.0009 & 2.0002 & 2.0001 \\
    			&                      & $E_{-1}^{h_0,\tau}$ & 4.5752E-05 & 1.1432E-05 & 2.8575E-06 & 7.1436E-07 \\
    			&                      & rate & $*$ & 2.0008 & 2.0002 & 2.0001 \\
    			&                      & $E_{-2}^{h_0,\tau}$ & 7.3626E-05 & 1.8395E-05 & 4.5980E-06 & 1.1495E-06 \\
    			&                      & rate & $*$ & 2.0009 & 2.0002 & 2.0001 \\ \cmidrule{2-7}
    			& \multirow{9}{*}{\raisebox{-3ex}{TS4}} & $E_2^{h_0,\tau}$ & 5.8017E-07 & 3.7139E-08 & 2.3356E-09 & 1.4621E-10 \\
    			&                      & rate & $*$ & 3.9655 & 3.9911 & 3.9976 \\
    			&                      & $E_1^{h_0,\tau}$ & 3.2720E-07 & 2.1001E-08 & 1.3215E-09 & 8.2668E-11 \\
    			&                      & rate & $*$ & 3.9617 & 3.9902 & 3.9987 \\
    			&                      & $E_0^{h_0,\tau}$ & 5.5258E-07 & 3.5367E-08 & 2.2241E-09 & 1.3925E-10 \\
    			&                      & rate & $*$ & 3.9657 & 3.9911 & 3.9975 \\
    			&                      & $E_{-1}^{h_0,\tau}$ & 3.2671E-07 & 2.0969E-08 & 1.3195E-09 & 8.2570E-11 \\
    			&                      & rate & $*$ & 3.9617 & 3.9902 & 3.9983 \\
    			&                      & $E_{-2}^{h_0,\tau}$ & 5.8010E-07 & 3.7134E-08 & 2.3353E-09 & 1.4621E-10 \\
    			&                      & rate & $*$ & 3.9655 & 3.9911 & 3.9975 \\
    			\toprule
    			\multirow{11}{*}{\raisebox{-12ex}{\begin{tabular}[c]{@{}c@{}}Spatial \\ direction\end{tabular}}}
    			& & $h$ & 1  & 1/2  & 1/4   & 1/8  \\ \cmidrule{2-7}
    			& \multirow{5}{*}{\raisebox{-2.5ex}{TS2}} & $E_2^{h,\tau_0}$ & 1.5025E-02 & 2.1334E-04 & 8.0943E-10 & 5.0698E-13 \\ [2.5pt]
    			&                      & $E_1^{h,\tau_0}$ & 1.0510E-02 & 8.4964E-05 & 3.9352E-10 & 4.9918E-13 \\ [2.5pt]
    			&                      & $E_0^{h,\tau_0}$ & 1.4270E-02 & 2.0077E-04 & 7.5453E-10 & 5.0477E-13 \\ [2.5pt]
    			&                      & $E_{-1}^{h,\tau_0}$ & 1.0646E-02 & 8.1368E-05 & 3.7883E-10 & 5.1567E-13 \\ [2.5pt]
    			&                      & $E_{-2}^{h,\tau_0}$ & 1.5169E-02 & 2.1555E-04 & 8.0559E-10 & 5.1528E-13 \\ \cmidrule{2-7}
    			& \multirow{5}{*}{\raisebox{-2.5ex}{TS4}} & $E_2^{h,\tau_0}$ & 1.5025E-02 & 2.1334E-04 & 8.1061E-10 & 1.5136E-12 \\ [2.5pt]
    			&                      & $E_1^{h,\tau_0}$ & 1.0510E-02 & 8.4964E-05 & 3.9306E-10 & 1.5238E-12 \\ [2.5pt]
    			&                      & $E_0^{h,\tau_0}$ & 1.4270E-02 & 2.0077E-04 & 7.5567E-10 & 1.5391E-12 \\ [2.5pt]
    			&                      & $E_{-1}^{h,\tau_0}$ & 1.0646E-02 & 8.1368E-05 & 3.7846E-10 & 1.5625E-12 \\ [2.5pt]
    			&                      & $E_{-2}^{h,\tau_0}$ & 1.5169E-02 & 2.1555E-04 & 8.0677E-10 & 1.5388E-12 \\
    			\bottomrule
    	\end{tabular}}
    \end{table}

    Table \ref{Acc1} presents the temporal and spatial errors computed by TS2 and TS4 at time $t=0.5$ on the computational domain $\mathcal{D}=\left[-12,12\right]^2$ for the 2D case, while Table \ref{Acc2} presents those at time $t=0.3$ on the computational domain $\mathcal{D}=\left[-8,8\right]^3$ for the 3D case. Tables \ref{Acc1}-\ref{Acc2} demonstrate that TS2/TS4 achieves second/fourth order accuracy in time and exhibits spectral accuracy in space. Furthermore, a higher order operator splitting scheme is made possible owing to the exact integrability of both subproblems.

\subsection{Efficiency test}
  	In this subsection, we compare the computational costs of our method and that proposed in \cite{LY2025}. In addition, we illustrate the efficiency of our numerical method by characterizing the functional relationship between the computational costs and the discrete problem size. For convenience, we denote our method as \textbf{M1} and the reference method as \textbf{M2}.
  	
  	
    \begin{exmp}[\textbf{\textit{Efficiency}}]\label{efficiency}
    	In this example, we test the computational costs for different total grid number $N_{tot}:=N^d$ in both 2D and 3D cases, using M1 and M2 method. We set the parameters $c_0=10$, $c_1=1$, $c_2=1$, $\Omega=0.2$, $\gamma=0.1$, and the initial condition \eqref{ini1}. The computational cost is defined as the CPU time(s) consumed to advance the simulations from time $t=0$ to $t=0.1$. The simulations were coded in FORTRAN, and executed on the dual-socket 3.00GHz Intel(R) Xeon(R) Gold 6248R CPUs with a 35.75 MB L3 cache in Ubuntu 18.04.6 LTS with the Intel compiler ifort.
    \end{exmp}
    \begin{table}[h!]
    	\centering
    	\caption{\footnotesize Timing results (in seconds) with the total grid number $N_{tot}$ by M1 method for $d$-dimensional cases in \textbf{Example \ref{efficiency}}.}
    	\label{Eff1}
    	{\begin{tabular}{clcccc}
    			\toprule
    			& $N_{tot}$ & $64^d$ & $128^d$ & $192^d$ & $256^d$ \\
    			\midrule
    			\multirow{2}{*}{\centering $d=2$}
    			& TS2 & 0.68 & 2.99 & 8.12 & 17.77\\
    			& TS4 & 1.57 & 7.02 & 18.03 & 40.71\\
    			\midrule
    			\multirow{2}{*}{\centering $d=3$}
    			& TS2 & 75.22 & 622.01 & 2449.06 & 6137.38\\
    			& TS4 & 172.46 & 1643.12 & 5671.32 & 14688.66\\
    			\bottomrule
    	\end{tabular}}
    \end{table}
    \begin{table}[h!]
    	\centering
    	\caption{\footnotesize Timing results (in seconds) with the total grid number $N_{tot}$ by M2 method for $d$-dimensional cases in \textbf{Example \ref{efficiency}}.}
    	\label{Eff2}
       {\begin{tabular}{clcccc}
       		\toprule
       		& $N_{tot}$ & $64^d$ & $128^d$ & $192^d$ & $256^d$ \\
       		\midrule
       		\multirow{2}{*}{\centering $d=2$}
       		& TS2 & 0.65 & 2.73 & 7.04 & 15.99\\
       		& TS4 & 1.48 & 6.22 & 16.12 & 36.51\\
       		\midrule
       		\multirow{2}{*}{\centering $d=3$}
       		& TS2 & 69.28 & 564.38 & 2253.20 & 5663.66\\
       		& TS4 & 159.21 & 1516.94 & 5225.02 & 13462.51\\
       		\bottomrule
       \end{tabular}}
    \end{table}
    \begin{figure}[h]
    	\centering
    	\includegraphics[scale=0.43]{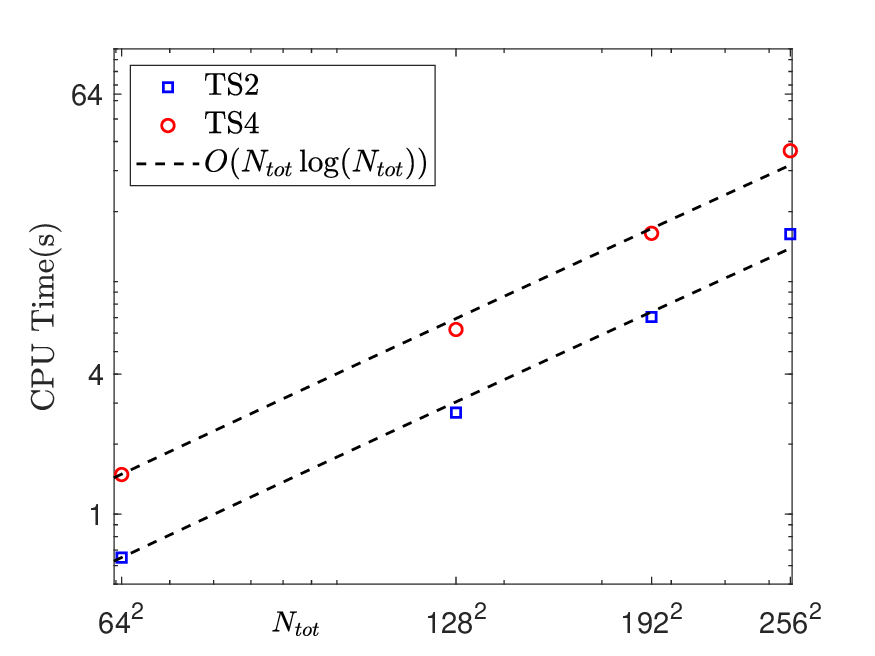}
    	\includegraphics[scale=0.43]{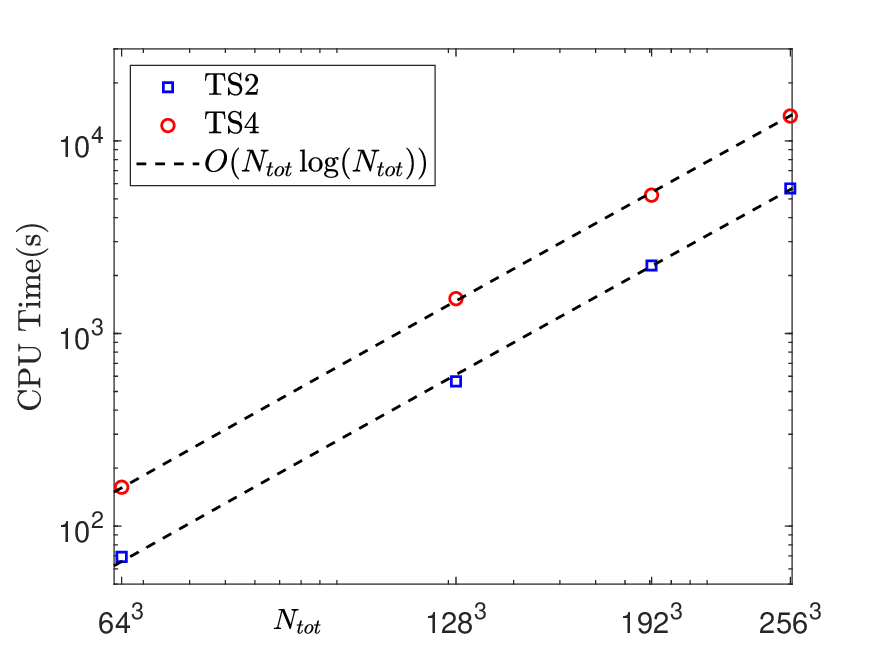}
    	\caption{Log-Log plots of timing results with the total grid $N_{tot}$ by TS2 and TS4 for both 2D (left) and 3D (right) cases in \textbf{Example \ref{efficiency}}.}
    	\label{Eff3}
    \end{figure}

    Table \ref{Eff1} presents the computational costs obtained by the M1 method using TS2 and TS4 with a time step $\tau = 10^{-3}$ on the computational domain $\mathcal{D} = [-12, 12]^d$, for various total grid sizes $N_{tot}$ in both 2D and 3D cases. Table \ref{Eff2} reports the corresponding results obtained by the M2 method under the same setting. Figure \ref{Eff3} illustrates the results from Table \ref{Eff2} via Log-Log plots.
	A comparison between Table \ref{Eff1} and Table \ref{Eff2} exhibits that the computational costs of M2 are less than that of M1, which aligns with Remark \ref{comput-cost}.
    As shown in Table \ref{Eff2} and Figure \ref{Eff3}, the computational cost scales approximately as $O(N_{\text{tot}}\log N_{\text{tot}})$, which agrees with the theoretical analysis in Remark \ref{NlogN} and thus confirms the efficiency of our numerical method.

\subsection{Dynamical properties verification.}
    In this subsection, we apply our numerical method to investigate the dynamical laws, including the conservation of mass, energy, and magnetization, as well as the evolution of angular momentum expectation and condensate widths.

    \begin{exmp}\label{properties1}
    	In this example, we numerically characterize the dynamical laws in 2D case. To this end, we take the parameters $c_0 = 120$, $c_1 = 1$, $c_2 = 1$, $\Omega = 0.2$, adopt the harmonic potential $V(\mathbf{x})$ given in \eqref{harmonic}, and consider the following three cases
    	\begin{itemize}
    		\item $\mathrm{\mathbf{Case~i:}}$ $\gamma_x=\gamma_y=1$, $\gamma=0$.
    		\item $\mathrm{\mathbf{Case~ii:}}$ $\gamma_x=\gamma_y=1$, $\gamma=0.9$.
    		\item $\mathrm{\mathbf{Case~iii:}}$ $\gamma_x=1.3$, $\gamma_y=1$, $\gamma=0$.
    	\end{itemize}
    	The initial functions are chosen as
    	\begin{equation}
    		\psi_{\pm2}^0(\mathbf{x})=\phi(\mathbf{x}),\quad\psi_{\pm1}^0(\mathbf{x})=\phi(\mathbf{x}),\quad\psi_0^0(\mathbf{x})=6\sqrt{7}\phi(\mathbf{x}),
    	\end{equation}
    	where $\phi(\mathbf{x})=e^{-|\mathbf{x}|^2/8}/\left(32\sqrt{\pi}\right)$.
    	The computational domain, computational time interval, mesh size and time step are respectively set as $\mathcal{D}=\left[-24,24\right]^2$, $t\in[0,12]$, $h=1/16$ and $\tau=10^{-3}$.
    \end{exmp}

    \begin{figure}[h]
    	\centering
    	\includegraphics[scale=0.43]{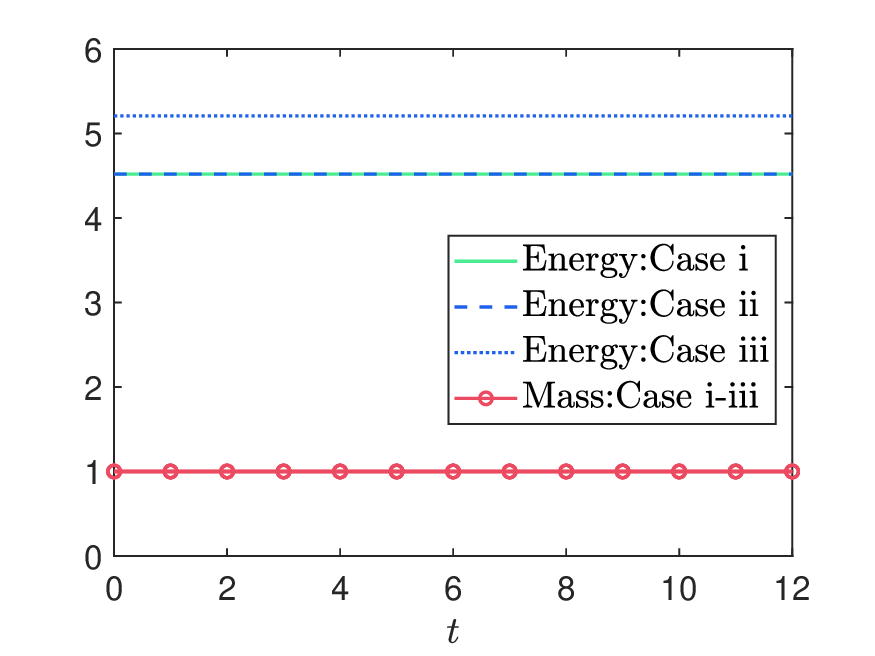}
    	\hspace{-0.3cm}
    	\includegraphics[scale=0.43]{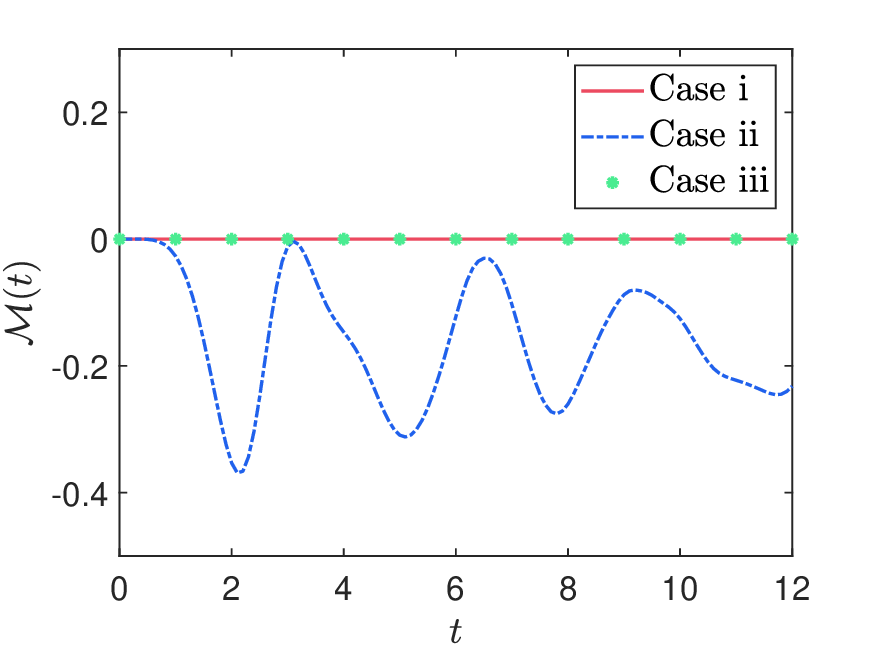}
    	\caption{Evolution of mass $\mathcal{N}(t)$ and energy $\mathcal{E}(t)$ (left) and magnetization $\mathcal{M}(t)$ (right) for Case i-Case iii in \textbf{Example \ref{properties1}}.}
    	\label{Proper1}
    \end{figure}
    \begin{figure}[h]
    	\centering
    	\includegraphics[scale=0.43]{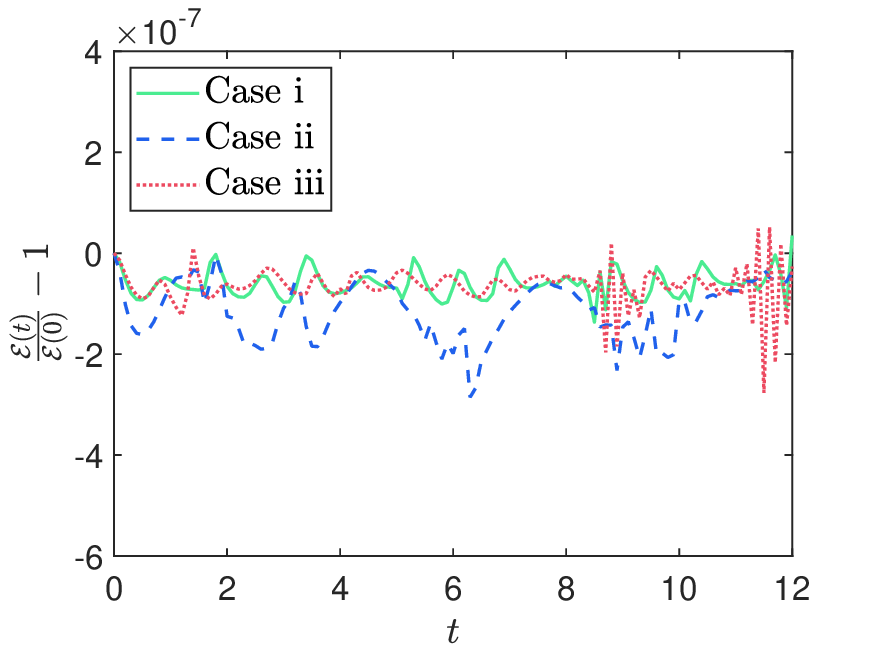}
    	\hspace{-0.3cm}
    	\includegraphics[scale=0.43]{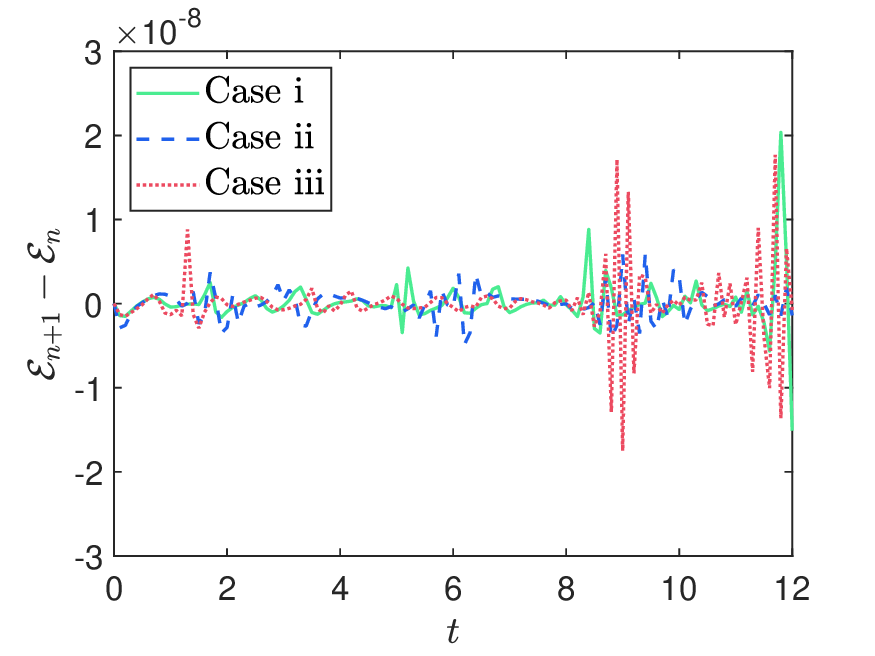}
    	\caption{Evolution of $\frac{\mathcal{E}(t)}{\mathcal{E}(0)}-1$ (left) and $\mathcal{E}_{n+1}-\mathcal{E}_{n}$ (right) for Case i-Case iii in \textbf{Example \ref{properties1}}.}
    	\label{Proper2}
    \end{figure}
    \begin{figure}[h]
    	\centering
    	\setcounter {subfigure} {0} a){
    		\includegraphics[scale=0.4]{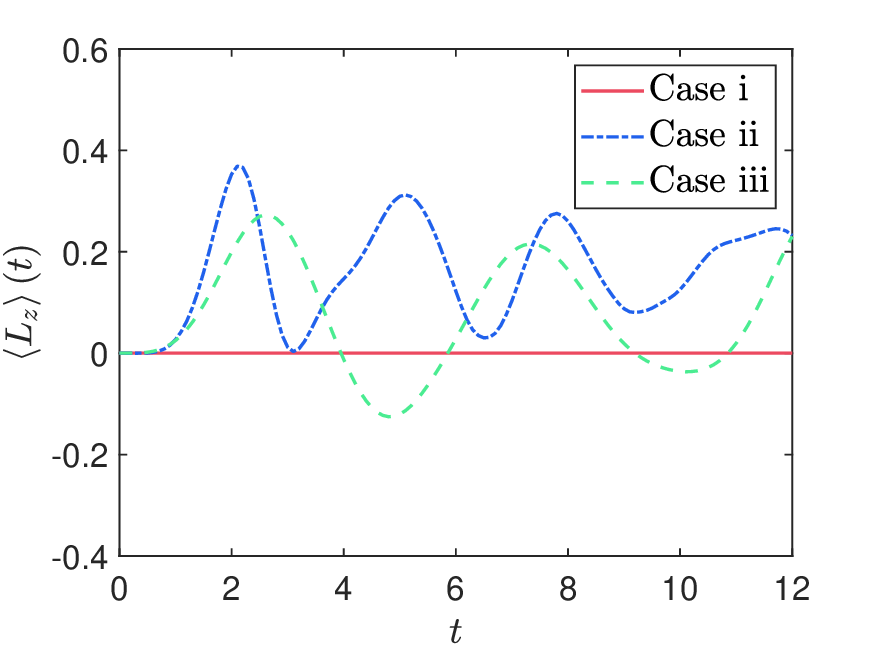}}
    	\hspace{-0.3cm}
    	\setcounter {subfigure} {0} b){
    		\includegraphics[scale=0.4]{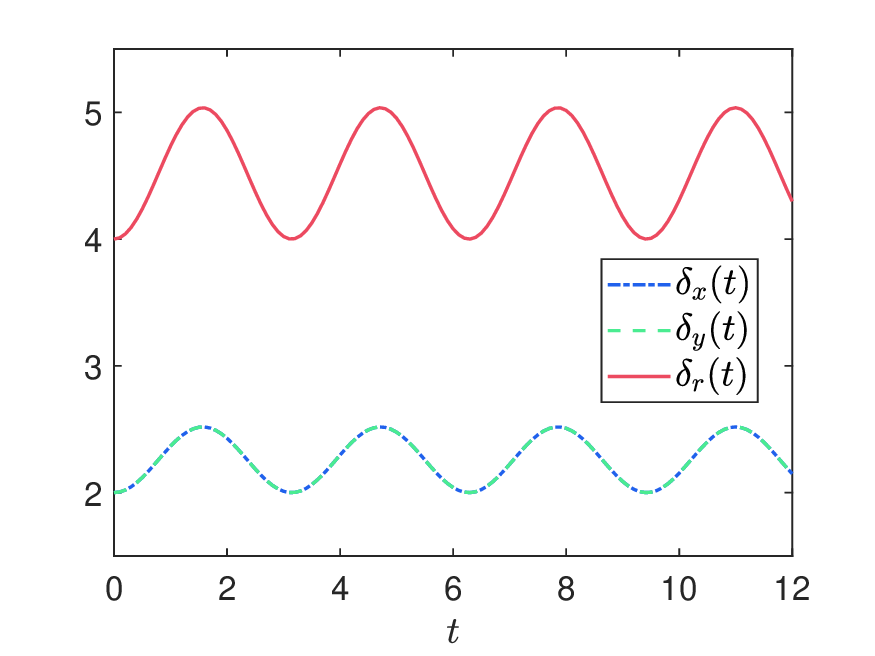}} \\
    	\setcounter {subfigure} {0} c){
    		\includegraphics[scale=0.4]{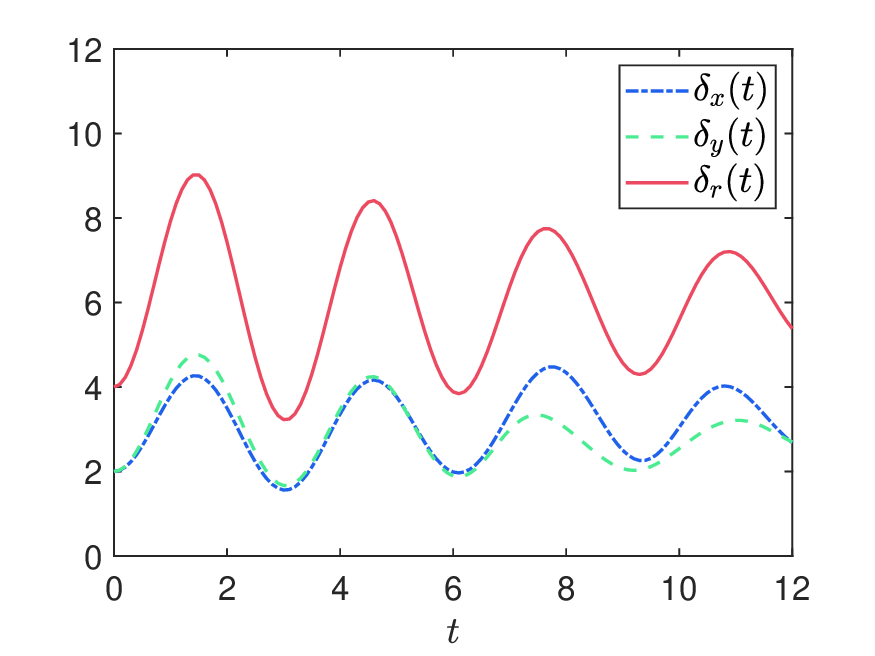}}
    	\hspace{-0.3cm}
    	\setcounter {subfigure} {0} d){
    		\includegraphics[scale=0.4]{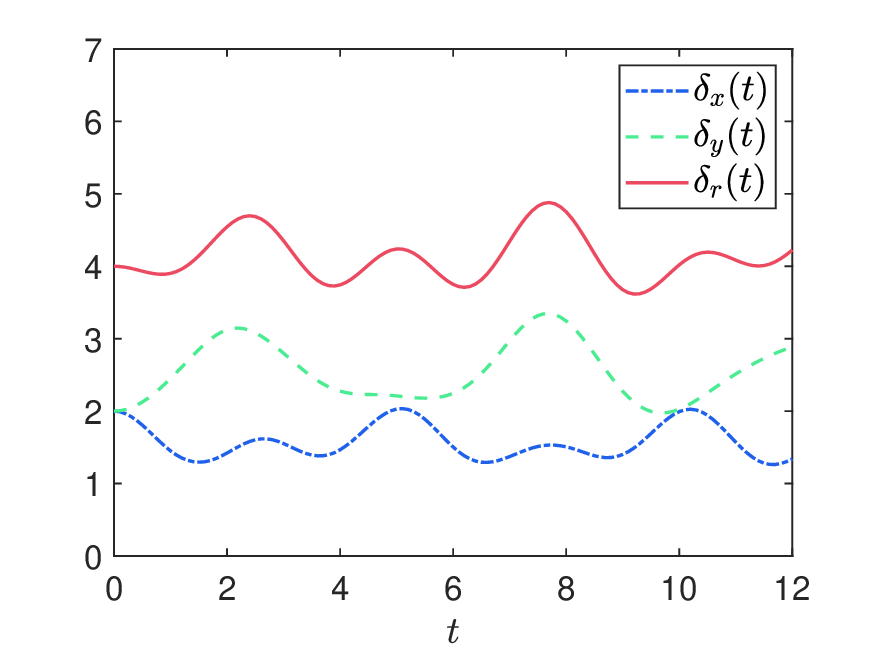}}
    	\caption{Evolution of angular momentum expectation $\langle L_z \rangle (t)$ (a) and condensate widths (b)-(d) for Case i-Case iii in \textbf{Example \ref{properties1}}.}
    	\label{Proper3}
    \end{figure}

    Figure \ref{Proper1}-\ref{Proper2} present the time evolution of mass $\mathcal{N}(t)$, energy $\mathcal{M}(t)$, magnetization $\mathcal{M}(t)$, $\mathcal{E}(t)/\mathcal{E}(0)-1$, and $\mathcal{E}_{n+1}-\mathcal{E}_n$.
    Figure \ref{Proper1} demonstrates the conservation of mass, as well as the conservation of magnetization under the condition $\gamma = 0$. Furthermore, Figure \ref{Proper1}-\ref{Proper2} show that energy is approximately conserved with high accuracy at the discrete level.
    Figure \ref{Proper3} presents the evolution of the angular momentum expectation $\langle L_z\rangle(t)$ and the $\delta_{\nu}(t)$ ($\nu=x$, $y$, $r$).
    Figure \ref{Proper3} shows that under the condition about radially symmetric harmonic potential and $\gamma=0$, the angular momentum expectation is conserved and $\delta_r$ is periodic. Furthermore, we observe that $\delta_x = \delta_y = \frac{1}{2}\delta_r$ for radially symmetric initial data.

\subsection{SOC effects.}
    In this subsection, we investigate the effects of SOC on dynamics in rotating spin-orbit coupled spin-2 BECs.
    \begin{exmp}\label{soc1}
    	In the example, we investigate the SOC effects in 2D case. To this end, we take the parameters $c_0=243$, $c_1=12.1$, $c_2=-13$ \cite{C2011,T2020}, $\Omega=0.2$, and choose different value of the spin-orbit coupling strength $\gamma$. Regarding the initial functions, we consider the following two cases
    	\begin{itemize}
    		\item $\mathrm{\mathbf{Case~i:}}$ $\quad\psi_{\pm2}^0(\mathbf{x})=\phi(\mathbf{x}),\quad\psi_{\pm1}^0(\mathbf{x})=\phi(\mathbf{x}),\quad\psi_0^0(\mathbf{x})=6\sqrt{7}\phi(\mathbf{x}).$
    		\item $\mathrm{\mathbf{Case~ii:}}$ $\quad\psi_{\pm2}^0(\mathbf{x})=\phi(\mathbf{x})(x+\im y),\quad\psi_{\pm1}^0(\mathbf{x})=\phi(\mathbf{x}),\quad\psi_0^0(\mathbf{x})=6\sqrt{7}\phi(\mathbf{x})(x+\im y),$
    	\end{itemize}
    	where $\phi(\mathbf{x})=e^{-|\mathbf{x}|^2/2}/\left(16\sqrt{\pi}\right)$.
    \end{exmp}
    \begin{figure}[h!]
    	\centering
    	\includegraphics[width=1\textwidth]{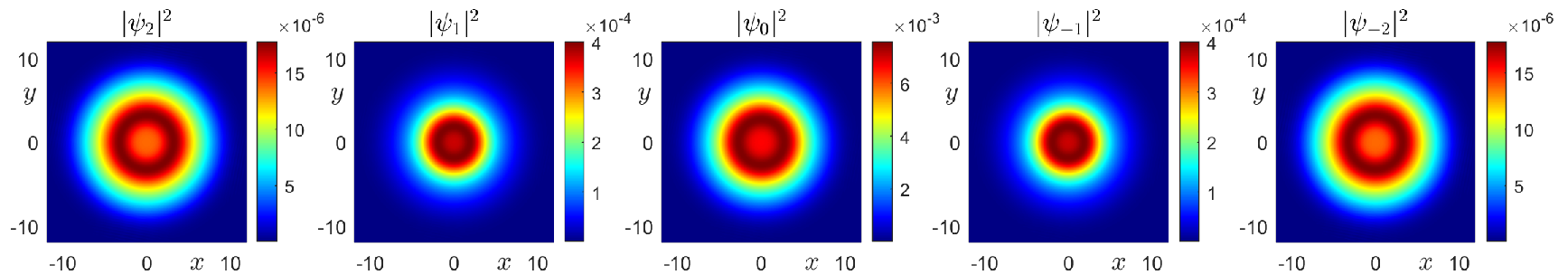}
    	\includegraphics[width=1\textwidth]{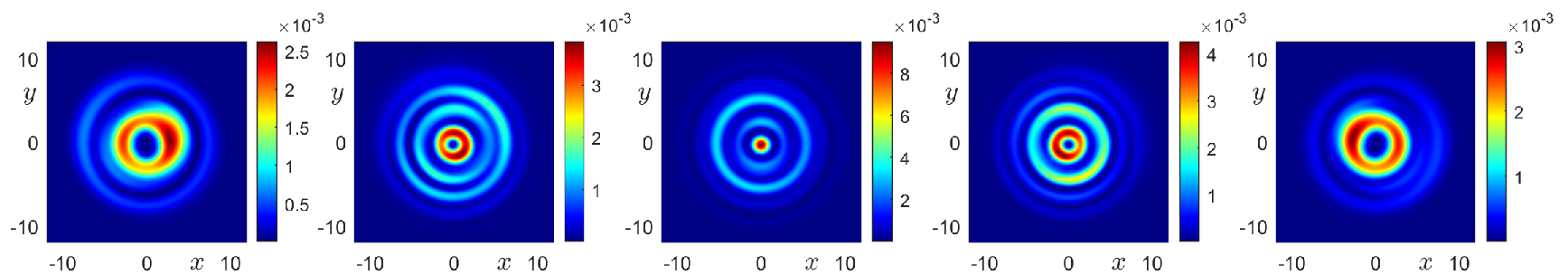} 
    	\includegraphics[width=1\textwidth]{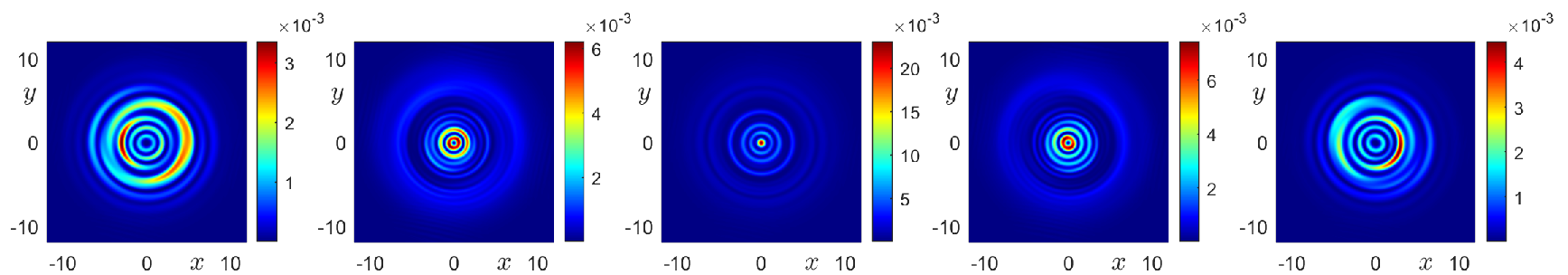} 
    	\caption{Contour plots of the densities with $\gamma = 0, 0.7, 2$ (top to bottom) in Case i of \textbf{Example \ref{soc1}}.}
    	\label{SOC1}
    \end{figure}
    \begin{figure}[h!]
    	\centering
    	\includegraphics[width=1\textwidth]{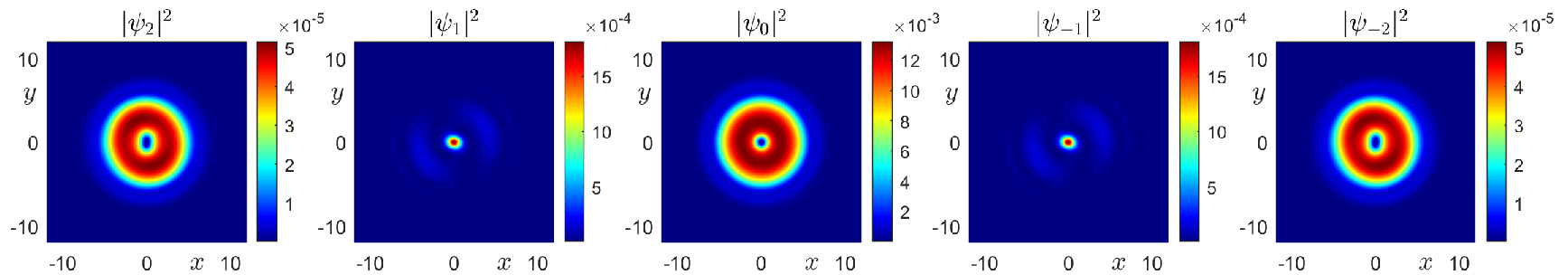}
    	\includegraphics[width=1\textwidth]{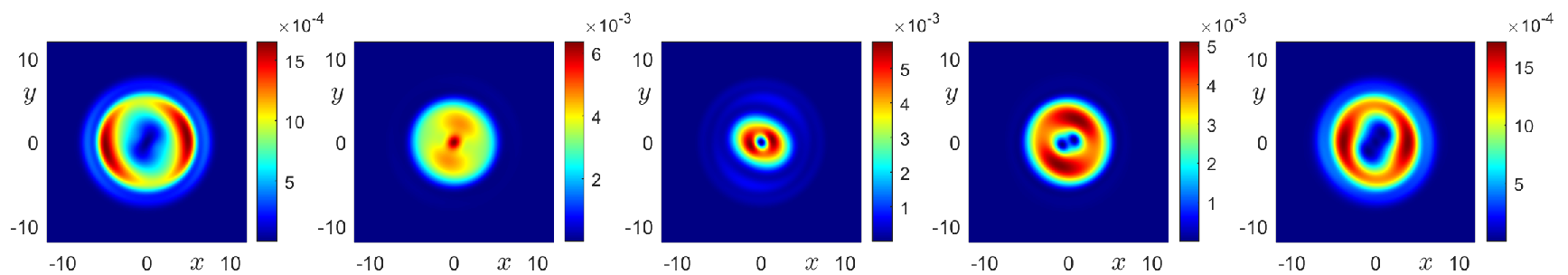} 
    	\includegraphics[width=1\textwidth]{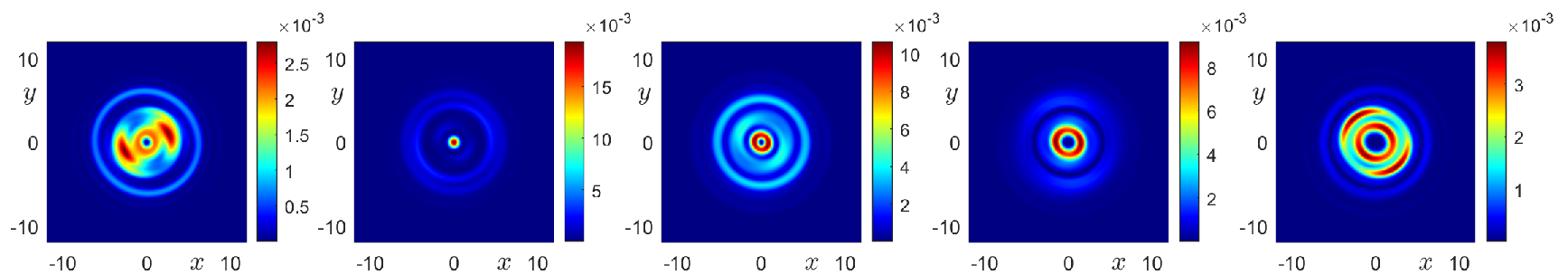} 
    	\caption{Contour plots of the densities with $\gamma = 0, 0.3, 1.2$ (top to bottom) in Case ii of \textbf{Example \ref{soc1}}.}
    	\label{SOC2}
    \end{figure}

    Figure \ref{SOC1}-\ref{SOC2} show contour plots of the densities with different $\gamma$ at time $t=1$ for Case i-ii respectively, computed using TS2 with mesh size $h=1/16$ and time step $\tau=10^{-3}$ on the computational domain $\mathcal{D}=[-24,24]^2$. These figures demonstrate that spin-orbit coupling interaction can generate spatial spin structures \cite{K2012}.

\subsection{Dynamics of a vortex lattice.}
    In this subsection, we investigate the dynamics of the vortex lattice in rotating spin-orbit coupled spin-2 BECs.
    \begin{exmp}\label{quan1}

        In this example, we choose the parameters $c_0=243$, $c_1=12.1$, $c_2=-13$, $\Omega=0.8$, $\gamma=2$, and a trapping potential \eqref{harmonic} with $\gamma_x=\gamma_y=1$. The initial condition is taken as a stationary vortex state, computed via the preconditioned conjugate gradient method under the conditions specified above. Then we investigate the dynamics of vortex lattice through the following two cases
    	\begin{itemize}
    		\item $\mathrm{\mathbf{Case~i:}}$ The spin-orbit coupling strength is increased to $\gamma = 4$.
    		\item $\mathrm{\mathbf{Case~ii:}}$ The external potential is made anisotropic with $\gamma_x = 0.9$ and $\gamma_y = 1.1$.
    	\end{itemize}
    \end{exmp}
    \begin{figure}[h!]
    	\centering
    	\includegraphics[width=1\textwidth]{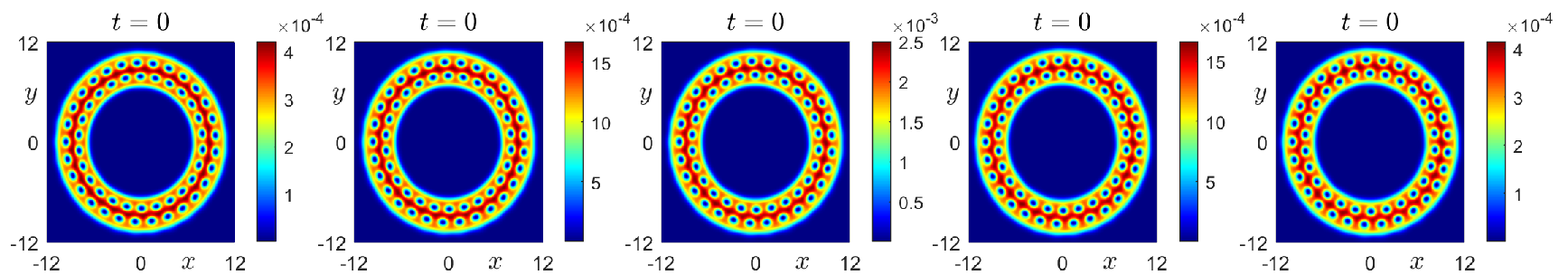} 
    	\includegraphics[width=1\textwidth]{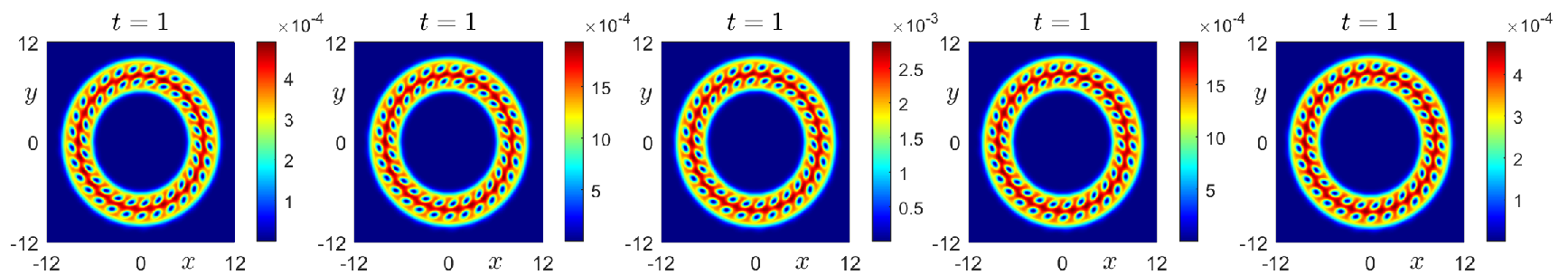} 
    	\includegraphics[width=1\textwidth]{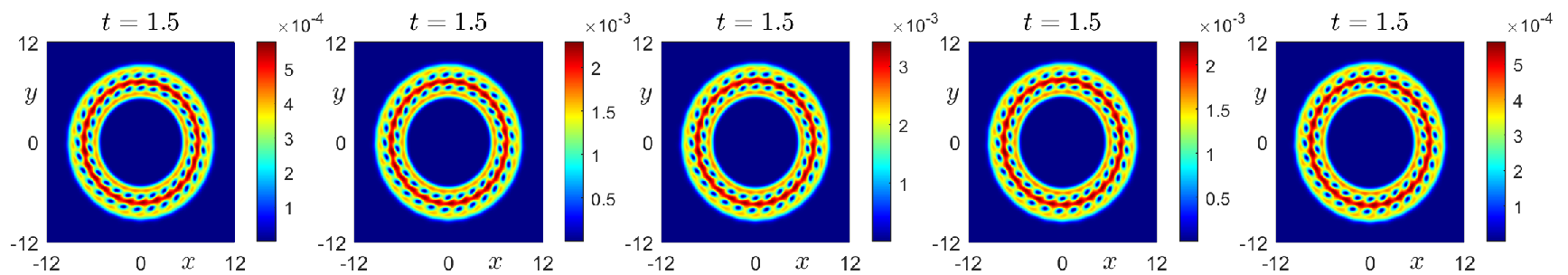} 
    	\caption{Contour plots of $|\psi_{\ell}|^2$ ($\ell = 2, 1, 0, -1, -2$, from left to right) for Case i in \textbf{Example \ref{quan1}}.}
    	\label{Quan1}
    \end{figure}
    \begin{figure}[h!]
    	\centering
    	\includegraphics[width=1\textwidth]{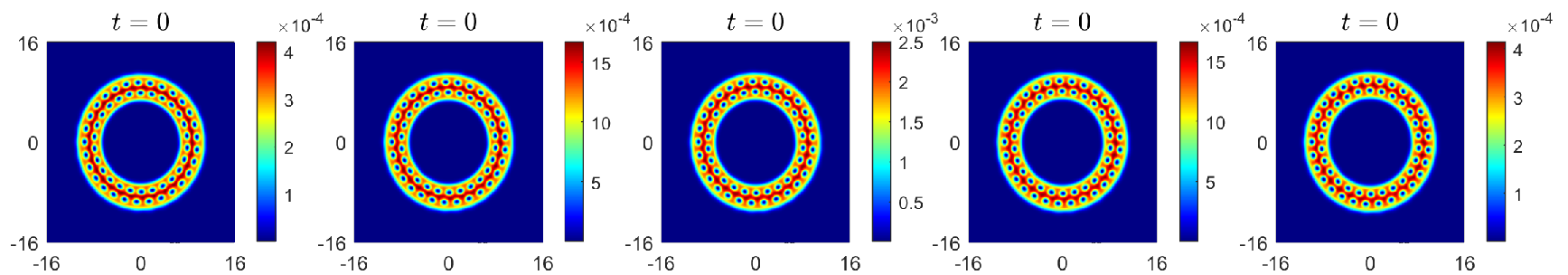} 
    	\includegraphics[width=1\textwidth]{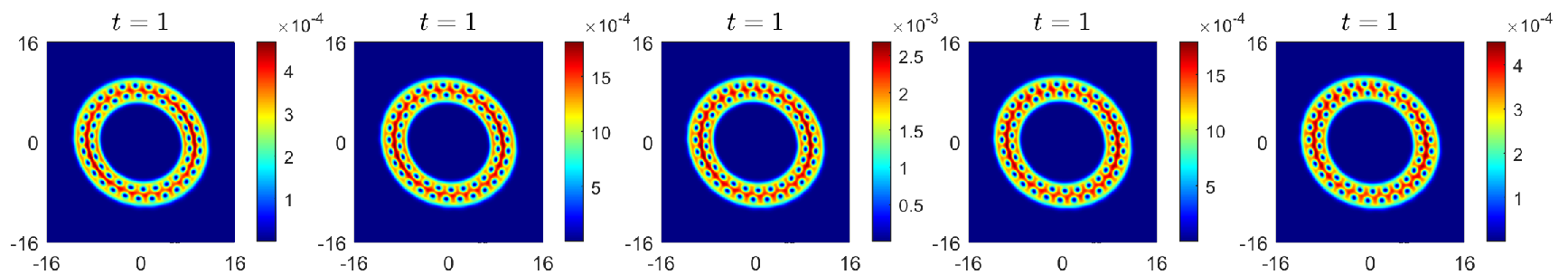} 
    	\includegraphics[width=1\textwidth]{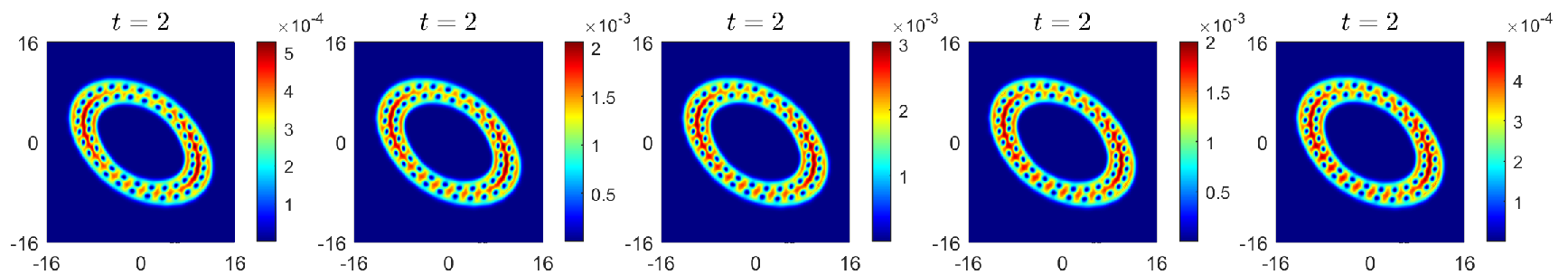} 
    	\caption{Contour plots of $|\psi_{\ell}|^2$ ($\ell = 2, 1, 0, -1, -2$, from left to right) for Case ii in \textbf{Example \ref{quan1}}.}
    	\label{Quan2}
    \end{figure}

    Figure \ref{Quan1}-\ref{Quan2} depict contour plots of the densities with different time $t$ for Case i-ii respectively, computed using TS2 with grid point number $N=256$ and time step $\tau=10^{-3}$ on the computational domain $\mathcal{D}=\left[-18,18\right]^2$.
    From Figure \ref{Quan1}, we can observe that the vortex lattice undergoes rotation and contraction as $\gamma$ increases. Additionally, Figure \ref{Quan2} shows that under the anisotropic external potential, the condensates expand along the $x$-direction and compress along the $y$-direction over time, leading to the formation of a sheet-like vortex lattice \cite{BW2006,E2002}.

\section{Conclusion}\label{sec-c}
  	We presented efficient high-order numerical schemes to simulate the dynamics of rotating SOC spin-2 BECs. The Hamiltonian is split into a linear part $\mathcal{A}$ (comprising the Laplace, rotation, and SOC terms) and a nonlinear part $\mathcal{B}$ (consisting of all remaining terms). We integrate the linear subproblem exactly and explicitly in phase space through a function mapping, which transforms the equation into an autonomous evolution equation without rotation term. This mapping The nonlinear subproblem is integrated analytically in physical space as usual. This compact splitting facilitates the design of high-order Fourier spectral method. Our method is spectrally accurate in space and high order in time. It is explicit, unconditionally stable, and conserves the mass and magnetization (when $\gamma = 0$) at discrete level. The dynamical laws of total mass, energy, magnetization, angular momentum expectation and condensate widths are also derived and confirmed numerically. Moreover, our method can be easily adapted and extended to simulate other rotating systems, such as the rotating SOC spin-3 BECs with/without dipole-dipole interactions.


\section*{Acknowledgements}



This work was partially supported by the National Natural Science Foundation of China No. 11971007 (Y. Yuan), No. 12271400, the National Key R\&D Program of China No. 2024YFA1012803 and basic research fund of Tianjin University under grant 2025XJ21-0010 (X. Liu and Y. Zhang).

\bibliographystyle{plain}

\end{document}